\documentclass[10pt]{amsart}
\usepackage[margin=2.5cm]{geometry}
\usepackage{tikz-cd}
\usepackage{graphicx}					
\usepackage{amssymb}
\usepackage{amsmath}
\usepackage{subfig, graphicx}
\usepackage{mathrsfs}
\usepackage{amssymb, amsthm, indentfirst}
\usepackage{relsize}
\usepackage{hyperref}
\usepackage{stmaryrd}
\usepackage{lineno}
\usepackage{bm}

\numberwithin{equation}{section}
\usepackage{color}
\theoremstyle{plain}
\newtheorem{thm}{Theorem}[section]

\newtheorem{lemma}[thm]{Lemma}

\newtheorem{corollary}[thm]{Corollary}
\theoremstyle{definition}

\newtheorem{rmk}[thm]{Remark}
\def\Gal{\operatorname{Gal}}

\newtheorem*{hypothesis*}{Hypothesis}

\title{Gamma factors for the Asai cube representation}
\author{SHIH-YU CHEN}

\address{Institute of Mathematics~\\Academia Sinica~\\ 6F, Astronomy-Mathematics Building, No.\,1, Sec.\,4, Roosevelt Road, Taipei 10617, R.O.C. (Taiwan)}

\email{sychen0626@gate.sinica.edu.tw}

\def\GL{{\rm{GL}}}
\def\GSp{{\rm GSp}}

\def\o{\frak{o}}

\def\A{{\mathbb A}}
\def\C{{\mathbb C}}
\def\R{{\mathbb R}}

\def\Z{{\mathbb Z}}

\def\<{\langle}
\def\>{\rangle}
\def\G{{\bf G}}

\def\bp{\begin{pmatrix}}
\def\ep{\end{pmatrix}}

\def\<{\langle}
\def\>{\rangle}

\def\GL{\operatorname{GL}}

\def\GSp{\operatorname{GSp}}

\def\1{\mathbf{1}}

\begin{document}
\maketitle

\begin{abstract}
We prove an equality between the gamma factors for the Asai cube representation of ${\rm R}_{E/F}\GL_2$ defined by the Weil$-$Deligne representations and the local zeta integrals of Ikeda and Piatetski-Shapiro$-$Rallis, where $E$ is an \'etale cubic algebra over a local field $F$ of characteristic zero. As an application we obtain the analytic properties of the automorphic $L$-functions for the Asai cube representation. 
\end{abstract}

\tableofcontents

\section{Introduction}

\subsection{Main results}
Let $F$ be a local field of characteristic zero and $E$ an \'etale cubic algebra over $F$. Denote by $\mathbb K$ the quadratic discriminant algebra of $E$. Let $\alpha$ be a basis of $E$ over $F$ and $\psi$ a non-trivial additive character of $F$. Let $\Pi$ be an irreducible generic admissible representation of $\GL_2(E)$. Denote by $\phi_\Pi : W_F' \rightarrow {}^L({\rm R}_{E/F}\GL_2)$ the $L$-parameter associated to $\Pi$ via the local Langlands correspondence. Here $W_F'$ is the Weil$-$Deligne group of $F$ and ${}^L({\rm R}_{E/F}\GL_2)$ is the Langlands dual group of ${\rm R}_{E/F}\GL_2$. We have the Asai cube representation (see \S\,\ref{S:Asai cubic Galois} for the precise definition)
\[
{\rm As} : {}^L({\rm R}_{E/F}\GL_2) \longrightarrow \GL(\C^2\otimes\C^2\otimes\C^2).
\]
Let $L(s,{\rm As}\,\Pi)$, $\varepsilon(s,{\rm As}\,\Pi,\psi)$, and $\gamma(s,{\rm As}\,\Pi,\psi)$ be the $L$-factor, $\varepsilon$-factor, and $\gamma$-factor, respectively, associated to the admisible representation ${\rm As}\circ\phi_\Pi$ of $W_F'$. On the other hand, associated to $\Pi$ we can define the local factors $L_{\rm PSR}(s,{\rm As}\,\Pi)$, $\varepsilon_{\rm PSR}(s,{\rm As}\,\Pi,\psi,\alpha)$, and $\gamma_{\rm PSR}(s,{\rm As}\,\Pi,\psi,\alpha)$ via the local zeta integrals of Ikeda and Piatetski-Shapiro$-$Rallis introduced in \cite{Ikeda1989} and \cite{PSR1987}. The local zeta integrals are of the form
\[
Z_{\tilde{\alpha}}(f_s,W) = \int_{F^\times U_0(F)\backslash \G(F)}f_s(\iota_{\tilde{\alpha}}(\eta g))W(g)\,dg.
\]
Here, $\tilde{\alpha}$ is a symplectic basis of the symplectic space $V(F) = E \oplus E$ equipped with the trace form, $\iota_{\tilde{\alpha}}$ is the isomorphism between $\GSp(V)(F)$ and $\GSp_6(F)$ induced by $\tilde{\alpha}$, $\G(F) = \{g \in \GL_2(E)\,\vert\,\det(g) \in F^\times\}\subseteq \GSp(V)(F)$, $U_0(F)$ is a unipotent subgroup of $\G(F)$, $\eta \in \GSp(V)(F)$ satisfies certain conditions, $f_s$ is a section in certain degenerate principal series representation of $\GSp_6(F)$, and $W$ is a Whittaker function of $\Pi$ with respect to $\psi_E = \psi\circ{\rm tr}_{E/F}$.
We refer to \S\,\ref {S:Asai cubic PSR} for more detail.

The main results of this paper are as follows:
\begin{thm}\label{T:1}
We have
\[
\gamma_{\rm PSR}(s,{\rm As}\,\Pi,\psi,\alpha) = \omega_\Pi(\Delta_{E/F}(\alpha))|\Delta_{E/F}(\alpha)|_F^{2s-1}\omega_{\mathbb{K}/F}(-1)\gamma(s,{\rm As}\,\Pi,\psi).
\]
Here $\omega_\Pi$ is the central character of $\Pi$, $\Delta_{E/F}(\alpha)$ is the relative discriminant of $\alpha$ for $E/F$, and $\omega_{\mathbb{K}/F}$ is the quadratic character associated to $\mathbb{K}/F$ by local class field theory.
\end{thm}

\begin{rmk}
When $E$ is unramified over $F$ and $\Pi$ is unramified, the identity follows from the explicit calculation due to Piatetski-Shapiro and Rallis in \cite{PSR1987}. When $E= F \times F \times F$, the identity was established by Ikeda and Ramakrishnan in \cite{Ikeda1989} and \cite{Rama2000}, respectively. When $E = F' \times F$ for some quadratic extension $F'$ of $F$, the assertion was partially proved by Chen$-$Cheng$-$Ishikawa in \cite{CCI2018}.
\end{rmk}

\begin{corollary}\label{C:1}
Assume $\max\{|L(\Pi)| , |L(\Pi^\vee)|\} < 1/2$. Then
\begin{align*}
L_{\rm PSR}(s,{\rm As}\,\Pi) &= L(s,{\rm As}\,\Pi),\\
\varepsilon_{\rm PSR}(s,{\rm As}\,\Pi,\psi,\alpha) &= \omega_\Pi(\Delta_{E/F}(\alpha))|\Delta_{E/F}(\alpha)|_F^{2s-1}\omega_{\mathbb{K}/F}(-1)\varepsilon(s,{\rm As}\,\Pi,\psi).
\end{align*}
Here $L(\Pi)$ and $L(\Pi^\vee)$ are defined in (\ref{E:parameter}).
\end{corollary}

Now we switch to a global setting. Let $F$ be a number field and $E$ an \'etale cubic algebra over $F$. Denote by $\A_E$ and $\A_F$ the rings of adeles of $E$ and $F$, respectively. Let $\psi$ be a non-trivial additive character of $\A_F/F$. Let $\Pi$ be an irreducible unitary cuspidal automorphic representation of $\GL_2(\A_E)$ with central character $\omega_{\Pi}$. Write $\omega = \omega_\Pi \vert_{\A_F^\times}$. Let
\[
L(s,{\rm As}\,\Pi) = \prod_{v}L(s,{\rm As}\,\Pi_v), \quad \varepsilon(s,{\rm As}\,\Pi) = \prod_v\varepsilon(s,{\rm As}\,\Pi_v,\psi_v)
\]
be the automorphic $L$-function and $\varepsilon$-factor associated to $\Pi$ and the Asai cube representation. 
We have the following three cases:
\[
\begin{cases}
E=F\times F \times F & \mbox{ Case 1},\\
E=F'\times F \mbox{ for some quadratic extension $F'$ of $F$}& \mbox{ Case 2},\\
E \mbox{ is a field} & \mbox{ Case 3}.
\end{cases}
\]
Combining Corollary \ref{C:1} and the results in \cite{Ikeda1992}, \cite{Rama2000}, and \cite{KS2002}, we have the following description of the analytic properties of $L(s,{\rm As}\,\Pi)$.

\begin{thm}\label{T:2}
The $L$-function $L(s,{\rm As}\,\Pi)$ is absolutely convergent for ${\rm Re}(s)\geq3/2$, admits meromorphic continuation to $s \in \C$, bounded in vertical strips of finite width, and satisfies the functional equation
\[
L(s,{\rm As}\,\Pi) = \varepsilon(s,{\rm As}\,\Pi) L(1-s,{\rm As}\,\Pi^\vee).
\]
If either $\omega^2$ is not principal or $\omega$ is principal, then $L(s,{\rm As}\,\Pi)$ is entire.
If $\omega^2$ is principal and $\omega$ is not principal, we may assume $\omega^2=1$ and $\omega \neq 1$ and let $K$ be the quadratic extension of $F$ associated to $\omega$ by class field theory. Then $L(s,{\rm As}\,\Pi)$ is not entire if and only if $K \neq F'$ in Case 2 and there exists a unitary Hecke character $\chi$ of $\A_{E\otimes_FK}^\times$ with $\chi \vert_{\A_K^\times} = 1$ such that 
$\Pi = {\rm Ind}_{E\otimes_FK}^{E}(\chi)$. In this case, we have
\[
L(s,{\rm As}\,\Pi) = \zeta_K(s)\begin{cases}
L(s,\chi_1^{-1}\chi_1^\sigma)L(s,\chi_2^{-1}\chi_2^\sigma)L(s,\chi_3^{-1}\chi_3^\sigma) & \mbox{ in Case 1 and $\chi=(\chi_1,\chi_2,\chi_3)$},\\
L(s,\chi_1^{-1}\chi_1^\sigma)L(s,\chi_2^{-1}\chi_2^\sigma) & \mbox{ in Case 2 and $\chi=(\chi_1,\chi_2)$},\\
L(s,\chi^{-1}\chi^\sigma) & \mbox{ in Case 3}.
\end{cases}
\]
Here ${\rm Ind}_{E\otimes_FK}^{E}(\chi)$ is the automorphic induction of $\chi$ from $\A_{E\otimes_FK}^\times$ to $\GL_2(\A_E)$, $\zeta_K(s)$ is the completed Dedekind zeta function of $K$, and $\sigma$ is the generator of 
\[
\begin{cases}
\Gal(K/F) & \mbox{ in Case 1},\\
\Gal(F'K/F') & \mbox{ in Case 2},\\
\Gal(EK/E) & \mbox{ in Case 3}.
\end{cases}
\]
\end{thm}


\subsection{An outline of the proof}
We sketch the proof of Theorem \ref{T:1}. Since the assertion is known when $F$ is archimedean, we assume $F$ is non-archimedean. Let $f_s$ be a good section of $I(\omega,s)$ and $W$ a Whittaker function of $\Pi$ with respect to $\psi_E$. By the functional equation of the local zeta integrals, it suffices to prove
\begin{align}\label{E:outline 1}
Z_{\tilde{\alpha}}(M_{\rm w}^*f_s,W^\vee) = \omega_\Pi(\Delta_{E/F}(\alpha))|\Delta_{E/F}(\alpha)|_F^{2s-1}\omega_{\mathbb{K}/F}(-1)\gamma(s,{\rm As}\,\Pi,\psi) Z_{\tilde{\alpha}}(f_s,W).
\end{align}
We extend $\Pi$ to a family of irreducible generic admissible representations $\Pi_\lambda$ of $\GL_2(E)$ defined in (\ref{E:holomorphic family}).
Here $M_{\rm w}^*$ is the normalized intertwining operator (\ref{E:io}), $W^\vee(g) = \omega_\Pi(\det(g))^{-1}W(g)$, and $\lambda$ varies in the domain $\mathcal{D}(\Pi)$ defined in (\ref{E:domain}). Write $\omega_\lambda = \omega_{\Pi_\lambda}\vert_{F^\times}$. Let $f_{s,\lambda}$ be a good section of $I(\omega_\lambda,s)$ and  $W_\lambda$ a holomorphic family of Whittaker functions of $\Pi_\lambda$ with respect to $\psi_E$ extending $f_s$ and $W$, respectively (cf.\,\S\,\ref{S:Whittaker}). Let 
\[
\mathcal{Z}_1(s,\lambda) = \frac{Z_{\tilde\alpha}(M_{\rm w}^*f_{-s+1/2,\lambda},W_\lambda^\vee)}{L(s+1/2,{\rm As}\,\Pi_\lambda^\vee)}
\]
and
\[
\mathcal{Z}_2(s,\lambda) = \omega_{\Pi_\lambda}(\Delta_{E/F}(\alpha))|\Delta_{E/F}(\alpha)|_F^{2s}\omega_{\mathbb{K}/F}(-1)\varepsilon(s+1/2,{\rm As}\,\Pi_\lambda,\psi,\alpha)\frac{Z_{\tilde \alpha}(f_{s+1/2,\lambda},W_\lambda)}{L(s+1/2,{\rm As}\,\Pi_\lambda)}
\]
be meromorphic functions on $\C \times \mathcal{D}(\Pi)$.
By the uniform asymptotic estimate for $W_\lambda$ proved in Lemma \ref{L:analytic family 2}, we show in Lemma \ref{L:convergence PSR} below that $\mathcal{Z}_1$ and $\mathcal{Z}_2$ define holomorphic functions on the domain
\[
\{(s,\lambda) \in \C \times \mathcal{D}(\Pi) \,\vert\, {\rm Re}(s) > |\lambda|_\Pi-1/2\}.
\]
Here $|\lambda|_\Pi \in \R_{ \geq 0}$ is the absolute value of $\lambda$ with respect to $\Pi$ defined in (\ref{E:parameter}). Now we use the limit multiplicity method, which is a global-to-local argument (for example, see \cite[\S\,3.8]{Raphael2018} and \cite[\S\,5.3]{CI2019}). More precisely, based on the following ingredients: 
\begin{itemize}
\item the limit multiplicity property for the principal congruence subgroups of $\GL_2$ proved in \cite{FLM2015}, \item the known cases for (\ref{E:outline 1}) recalled in Lemma \ref{L:unramified gamma} and Corollary \ref{L:unramified gamma2},
\item the equality between the Asai cube $\gamma$-factors defined by the Weil$-$Deligne representation and the Langlands$-$Shahidi method proved in \cite{HL2018},
\end{itemize}
we prove that the functional equation
\begin{align}\label{E:outline 2}
\mathcal{Z}_1(-s,\lambda) = \mathcal{Z}_2(s,\lambda)
\end{align}
holds for $s \in \C$ and $\lambda$ in a dense subset of 
\[
\{\lambda \in \mathcal{D}(\Pi)\,\vert\,\Pi_\lambda\mbox{ is tempered}\}.
\]
Note that $\Pi_\lambda$ is tempered if and only if $|\lambda|_\Pi=0$. It then follows from the continuity and holomorphicity that $\mathcal{Z}_1$ and $\mathcal{Z}_2$ define holomorphic functions on the domain
\[
\C \times \{\lambda \in \mathcal{D}(\Pi)\,\vert\,|\lambda|_\Pi<1/2\}
\]
and satisfy the functional equation (\ref{E:outline 2}). Finally, we are in the position to apply \cite[Proposition 2.8.1]{Raphael2018}, which together with Lemma \ref{L:convergence PSR} imply that $\mathcal{Z}_1$ and $\mathcal{Z}_2$ admit holomorphic continuation to $\C \times \mathcal{D}(\Pi)$ and satisfy the functional equation (\ref{E:outline 2}). In particular, (\ref{E:outline 1}) follows.

\subsection{Notation}
Let $F$ be a local field of characteristic zero. When $F$ is non-archimedean, let $\o_F$ be the ring of integers of $F$, $\varpi_F$ a uniformizer of $\o_F$, $q_F$ the cardinality of $\o_F / \varpi_F \o_F$, $|\mbox{ }|_F$ the absolute value on $F$ normalized so that $|\varpi_F|_F = q_F^{-1}$, and ${\rm ord}_F$ the valuation on $F$ normalized so that ${\rm ord}_F(\varpi_F)=1$. When $F$ is archimedean, let $|\mbox{ }|_{\R}=|\mbox{ }|$ be the usual absolute value on $\R$ and $|z|_\C=z\overline{z}$ on $\C$.

An additive character $\psi$ of $F$ is a continuous homomorphism $\psi : F \rightarrow \C^{\times}$. For $a \in F^{\times}$, let $\psi^a$ be the additive character defined by $\psi^a(x)=\psi(ax)$.

A character $\chi$ of $F^{\times}$ is a continuous homomorphism $\chi : F^{\times}\rightarrow \C^{\times}$. For a character $\chi$ of $F^\times$, let ${\rm wt}(\chi) \in \R$ defined so that $|\chi| = |\mbox{ }|_F^{{\rm wt}(\chi)}$ and denote by $L(s,\chi)$, $\varepsilon(s,\chi,\psi)$, and $\gamma(s,\chi,\psi)$ 
the $L$-factor, $\varepsilon$-factor, and $\gamma$-factor of $\chi$, respectively, with respect to an additive character $\psi$ of $F$ defined in \cite{Tate1979}.

Let $B$ be the standard Borel subgroup of $\GL_2$ consisting of upper triangular matrices and $N$ its unipotent radical. We put 
\[{\bf a}(\nu) = \bp \nu & 0 \\ 0 & 1 \ep,\quad {\bf d}(\nu) =   \bp 1 & 0 \\ 0 & \nu \ep,\quad {\bf m}(t)= \bp t & 0 \\ 0 & t^{-1}\ep,\quad {\bf n}(x) = \bp 1 & x \\ 0 & 1\ep,\quad w = \bp 0 & 1 \\ -1 & 0\ep\]
for $\nu,t \in \mathbb{G}_m$ and $x \in \mathbb{G}_a$.

\section{Holomorphic family of Whittaker functions}\label{S:Whittaker}

Let $F$ be a local field of characteristic zero and $\psi$ a non-trivial additive character of $F$. Let
\[
K=\begin{cases}
\GL_2(\o_F) & \mbox{ if $F$ is non-archimedean},\\
{\rm O}(2) & \mbox{ if $F=\R$},\\
{\rm U}(2) & \mbox{ if $F=\C$},
\end{cases}
\]
be a maximal compact subgroup of $\GL_2(F)$.
Denote by $C_0^{\infty}(N(F)\backslash \GL_2(F),\psi)$ the space of smooth functions $W: \GL_2(F) \rightarrow \C$ such that
\begin{itemize}
\item For $x \in F$ and $g \in \GL_2(F)$,
\[
W({\bf n}(x)g) = \psi(x)W(g).
\]
\item $W$ is right $K$-finite.
\end{itemize}
Let $\pi$ be an irreducible generic admissible representation of $\GL_2(F)$ with central character $\omega_\pi$. We denote by $\pi^\vee$ the contragredient representation of $\pi$ and by $\mathcal{W}(\pi,\psi)$ the Whittaker model of $\pi$ with respect to $\psi$.
Recall that $\mathcal{W}(\pi,\psi)$ is the image of a non-zero intertwining map $\pi \rightarrow C_0^{\infty}(N(F)\backslash \GL_2(F),\psi)$.
For $W \in \mathcal{W}(\pi,\psi)$, we define $W^\vee \in \mathcal{W}(\pi^\vee,\psi)$ by
\[
W^\vee(g) = \omega_\pi(\det(g))^{-1}W(g).
\]
We define $l(\pi) \in \R$ by
\[
l(\pi) = \begin{cases}
{\rm wt}(\omega_\pi)/2 & \mbox{ if $\pi$ is essentially square-integrable},\\
\min\{{\rm wt}(\chi_1),{\rm wt}(\chi_2)\} & \mbox{ if $\pi = {\rm Ind}_{B(F)}^{\GL_2(F)}(\chi_1 \otimes \chi_2)$}.
\end{cases}
\]
If $\pi$ is essentially square-integrable and $\lambda \in \C$, we define $\pi_\lambda = \pi \otimes |\mbox{ }|_F^{\lambda}$. If $\pi = {\rm Ind}_{B(F)}^{\GL_2(F)}(\chi_1 \otimes \chi_2)$ and $\lambda = (\lambda_1,\lambda_2) \in \C^2$, we define
\[
\pi_\lambda = {\rm Ind}_{B(F)}^{\GL_2(F)}(\chi_1|\mbox{ }|_F^{\lambda_1} \otimes \chi_2|\mbox{ }|_F^{\lambda_2}).
\]
Let $\mathcal{D}(\pi)$ be the domain associated to $\pi$ defined by
\[
\mathcal{D}(\pi) = \begin{cases}
\C & \mbox{ if $\pi$ is essentially square-integrable},\\
\{\lambda \in \C^2 \,\vert\, \pi_\lambda \mbox{ is irreducible}\} & \mbox{ if $\pi$ is a principal series representation}.
\end{cases}
\]
For $\lambda \in \mathcal{D}(\pi)$, define $|\lambda|_\pi \in \R_{\geq 0}$ by
\[
|\lambda|_\pi = \begin{cases}
|{\rm wt}(\omega_\pi)/2+{\rm Re}(\lambda)| & \mbox{ if $\pi$ is essentially square-integrable},\\
\max\{|{\rm wt}(\chi_1)+{\rm Re}(\lambda_1)|,|{\rm wt}(\chi_2)+{\rm Re}(\lambda_2)|\} & \mbox{ if $\pi = {\rm Ind}_{B(F)}^{\GL_2(F)}(\chi_1 \otimes \chi_2)$}.
\end{cases}
\]

We call a map
\[
\mathcal{D}(\pi) \longrightarrow C_0^{\infty}(N(F)\backslash \GL_2(F),\psi),\quad \lambda \longmapsto W_\lambda
\]
a holomorphic family of Whittaker functions of $\pi_\lambda$ with respect to $\psi$ if it satisfies the following conditions:
\begin{itemize}
\item The map $(\lambda,g) \mapsto W_\lambda(g)$ is continuous.
\item For each $g \in \GL_2(F)$, the map $\lambda \mapsto W_\lambda(g)$ is holomorphic on $\mathcal{D}(\pi)$.
\item For each $\lambda \in \mathcal{D}(\pi)$, the function $g \mapsto W_\lambda(g)$ belongs to $\mathcal{W}(\pi_\lambda,\psi)$.
\item $W_\lambda$ is right $K$-finite.
\end{itemize}
We recall a construction of holomorphic family of Whittaker functions. Let $\omega_\psi$ be the Weil representation of $\GL_2(F)$ on $\mathcal{S}(F^2)$, the space of Schwartz functions on $F^2$, with respect to $\psi$ defined by the following rules:
\begin{itemize}
\item For $t \in F^\times$,
\[
\omega_\psi({\bf m}(t))\varphi(x,y) = |t|_F\varphi(tx,ty).
\]
\item For $b \in F$,
\[
\omega_\psi({\bf n}(b))\varphi(x,y) = \psi(bxy)\varphi(x,y).
\]
\item 
\[
\omega_\psi(w)\varphi(x,y) = \int_{F^2}\varphi(u,v)\psi(uy+vx)\,du\,dv,
\]
where $du$ and $dv$ are self-dual with respect to $\psi$.
\item For $\nu \in F^\times$,
\[
\omega_\psi({\bf a}(\nu))\varphi(x,y) = \varphi(\nu x,y).
\]
\end{itemize}
Let
\begin{align*}
&\mathcal{S}_\psi(F^2) \\ 
&= \begin{cases}
\mathcal{S}(F^2) & \mbox{ if $F$ is non-archimedean},\\
\{P(x,y)e^{-\pi|a|_\R(x^2+y^2)} \,\vert\, P \in \C[x_1,x_2] \} &\mbox{ if $F=\R$ and $\psi(x)=e^{2\pi \sqrt{-1}\,ax}$},\\
\{P(x,\overline{x},y,\overline{y})e^{-2\pi|a|_\C^{1/2}(x\overline{x}+y\overline{y})} \,\vert\, P \in \C[x_1,x_2,x_3,x_4] \} &\mbox{ if $F=\C$ and $\psi(x)=e^{2\pi \sqrt{-1}\,{\rm tr}_{\C/\R}(ax)}$}.
\end{cases}
\end{align*}
If $\pi$ is essentially square-integrable and $W \in \mathcal{W}(\pi,\psi)$, then the map 
\begin{align}\label{E:holomorphic family DS}
\lambda \longmapsto W\cdot |\mbox{ }|_F^\lambda \circ \det
\end{align}
is a holomorphic family of Whittaker functions of $\pi_\lambda$ with respect to $\psi$. If $\pi = {\rm Ind}_{B(F)}^{\GL_2(F)}(\chi_1 \otimes \chi_2)$ and $\varphi \in \mathcal{S}_\psi(F^2)$, then the map
\begin{align}\label{E:holomorphic family PS}
\begin{split}
\lambda &\longmapsto W(\varphi,\lambda),\\
W(\varphi,\lambda)(g) &= \chi_1(\det(g))|\det(g)|_F^{\lambda_1+1/2} \int_{F^\times}\omega_\psi(g)\varphi(t,t^{-1}) \chi_1\chi_2^{-1}(t)|t|_F^{\lambda_1-\lambda_2}\,d^\times t
\end{split}
\end{align}
is a holomorphic family of Whittaker functions of $\pi_\lambda$ with respect to $\psi$. Note that for any fixed $\lambda_0 \in \mathcal{D}(\pi)$, any holomorphic family of Whittaker functions can be written as a linear combination, with holomorphic functions of $\lambda$ as coefficients, of holomorphic families of the form (\ref{E:holomorphic family DS}) or (\ref{E:holomorphic family PS}) in a neighborhood of $\lambda_0$.

We have the following uniform asymptotic estimate for holomorphic families of Whittaker functions.

\begin{lemma}\label{L:analytic family}
Assume $F$ is non-archimedean. Let $W_\lambda$ be a holomorphic family of Whittaker functions of $\pi_\lambda$ with respect to $\psi$. Let $\epsilon>0$. There exist an integer $n$ independent of $\epsilon,\lambda$ and a constant $C_{\lambda,\epsilon}>0$ bounded uniformly as $\lambda$ varies in a compact set such that
\[
|W_\lambda({\bf a}(\nu)k)| \leq C_{\lambda,\epsilon}\cdot \mathbb{I}_{\varpi_F^{-n}\o_F}(\nu)|\nu|_F^{l(\pi_\lambda)+1/2-\epsilon}
\]
for $\nu \in F^\times$ and $k \in \GL_2(\o_F)$.
\end{lemma}

\begin{proof}
Since we only consider the convergence for $\lambda$ varying in a compact set, it suffices to consider holomorphic families of the form (\ref{E:holomorphic family DS}) or (\ref{E:holomorphic family PS}).
We assume $\pi = {\rm Ind}_{B(F)}^{\GL_2(F)}(\chi_1 \otimes \chi_2)$ is a principal series representation. The other case follows from the asymptotic estimate for a single Whittaker function in \cite[Lemma 14.3]{JLbook2}. By the formulae defining $\omega_\psi$, there exist an integer $n$ and a constant $C>0$ such that
\[
|\omega_\psi({\bf a}(\nu)k)(x,y)| \leq C\cdot \mathbb{I}_{\varpi_F^{-n}\o_F \times \varpi_F^{-n}\o_F}(\nu x ,y)
\]
for $(x,y)\in F^2$, $\nu \in F^\times$, and $k \in \GL_2(\o_F)$. Let 
\[
s_1 = {\rm wt}(\chi_1),\quad s_2 = {\rm wt}(\chi_2).
\] 
Then we have
\begin{align*}
|W(\varphi,\lambda)({\bf a}(\nu)k)| &\leq C \cdot|\nu|^{s_1+{\rm Re}(\lambda_1)+1/2} \int_{F^\times} \mathbb{I}_{\varpi_F^{-n}\o_F \times \varpi_F^{-n}\o_F}(\nu t ,t^{-1})|t|_F^{s_1-s_2+{\rm Re}(\lambda_1-\lambda_2)}\,d^\times t\\
&= C \cdot\mathbb{I}_{\varpi^{-2n}\o_F}(\nu)|\nu|_F^{s_1+{\rm Re}(\lambda_1)+1/2} \sum_{m=-n-{\rm ord}_F(\nu)}^{n}q_F^{-s_1+s_2-{\rm Re}(\lambda_1-\lambda_2)}\\
&= C \cdot \mathbb{I}_{\varpi^{-2n}\o_F}(\nu) \left[ C_\lambda^{(1)}\cdot |\nu|_F^{s_1+{\rm Re}(\lambda_1)+1/2} + C_\lambda^{(2)}(\nu)\cdot |\nu|_F^{s_2+{\rm Re}(\lambda_2)+1/2} \right].
\end{align*}
for $\nu \in F^\times$ and $k \in \GL_2(\o_F)$. Here
\begin{align*}
C_\lambda^{(1)} & = \sum_{m=-n}^{n}q_F^{-s_1+s_2-{\rm Re}(\lambda_1-\lambda_2)},\\
C_\lambda^{(2)}(\nu) & =\begin{cases}
 \displaystyle{q_F^{n(s_1-s_2+{\rm Re}(\lambda_1-\lambda_2))}\left(\frac{1-|\nu|_F^{s_1-s_2+{\rm Re}(\lambda_1-\lambda_2)}}{1-q_F^{-s_1+s_2-{\rm Re}(\lambda_1-\lambda_2)}}\right) }& \mbox{ if $s_1+{\rm Re}(\lambda_1) \neq s_2 + {\rm Re}(\lambda_2)$},\\
 {\rm ord}_F(\nu) & \mbox{ if $s_1+{\rm Re}(\lambda_1) = s_2 + {\rm Re}(\lambda_2)$}.
 \end{cases}
\end{align*}
It is cleat that $C_\lambda^{(1)}$ is bounded uniformly as $\lambda$ varies in a compact set and there exists a constant $C_{\lambda,\epsilon}^{(2)}>0$ bounded uniformly as $\lambda$ varies in a compact set such that
\[
|C_{\lambda}^{(2)}(\nu)| \leq C_{\lambda,\epsilon}^{(2)} \cdot |\nu|_F^{-\epsilon}
\]
for $\nu \in \varpi_F^{-2n}\o_F$. This completes the proof.
\end{proof}

Let $E$ be a finite \'etale algebra of degree $d$ over $F$. Let $\psi_E$ the additive character of $E$ defined by $\psi_E = \psi \circ {\rm tr}_{E/F}$. We denote by $\mathcal{W}(\Pi,\psi_E)$ the Whittaker model of $\Pi$ with respect to $\psi_E$. Let $\Pi$ be an irreducible generic admissible representation of $\GL_2(E)$ with central character $\omega_\Pi$. Assume $E= F_1 \times \cdots \times F_r$ for some finite extension $F_i$ of degree $d_i$ over $F$. Then
\[
\Pi = \pi_1 \times \cdots \times \pi_r
\]
for some irreducible generic admissible representation $\pi_i$ of $\GL_2(F_i)$. Define $L(\Pi) \in \R$ by
\begin{align}\label{E:parameter}
L(\Pi) = \sum_{i=1}^r d_i\cdot l(\pi_i).
\end{align}
Let $\mathcal{D}(\Pi)$ be the domain associated to $\Pi$ defined by
\begin{align}\label{E:domain}
\mathcal{D}(\Pi) = \mathcal{D}(\pi_1)\times\cdots\times\mathcal{D}(\pi_r).
\end{align}
For $\lambda = (\lambda_1,\cdots,\lambda_r) \in \mathcal{D}(\Pi)$, define
\begin{align}\label{E:holomorphic family}
|\lambda|_\Pi  = \sum_{i=1}^{r}d_i \cdot |\lambda_i|_{\pi_i},\quad
\Pi_\lambda = (\pi_1)_{\lambda_1} \times \cdots \times (\pi_r)_{\lambda_r}.
\end{align}
Note that by definition we have
\[
|\lambda|_\Pi \geq \max\{|L(\Pi_\lambda)|,|L(\Pi_\lambda^\vee)|\}
\]
and $\Pi_\lambda$ is tempered if and only if $|\lambda|_\Pi=0$. Let
\begin{align}\label{E:imaginary domain}
\mathcal{D}(\Pi)^\circ = \{\lambda \in \mathcal{D}(\Pi)\,\vert\, |\lambda|_\Pi=0\}.
\end{align}
Similar to the case $E=F$, we have the notion of holomorphic families of Whittaker functions of $\Pi_\lambda$ with respect to $\psi_E$.

\begin{lemma}\label{L:analytic family 2}
Assume $F$ is non-archimedean. Let $W_\lambda$ be a holomorphic family of Whittaker functions of $\Pi_\lambda$ with respect to $\psi_E$. Let $\epsilon>0$. There exist an integer $n$ independent of $\epsilon,\lambda$ and a constant $C_{\lambda,\epsilon}>0$ bounded uniformly as $\lambda$ varies in a compact set such that
\[
|W_\lambda({\bf a}(\nu){\bf m}( t)k)| \leq C_{\lambda,\epsilon}\cdot \mathbb{I}_{\varpi_F^{-n}\o_F \times\cdots\times \varpi_F^{-n}\o_F}((\nu_1,\cdots,\nu_r))\prod_{i=1}^r|\nu_i|_F^{d_i \cdot l((\pi_i)_{\lambda_i})+d_i/2-\epsilon}
\]
for $\nu =  (\nu_1,\cdots,\nu_r)\in (F^\times)^r$, $t \in \mathcal{C}$, and $k \in \GL_2(\o_E)$. Here $\mathcal{C}$ is a complete set of coset representatives for $(F_1^\times \times \cdots \times F_r^\times) /( F^\times )^r$.
\end{lemma}

\begin{proof}
The assertion follows directly from Lemma \ref{L:analytic family} and the fact that $(F_1^\times \times \cdots \times F_r^\times) /( F^\times )^r$ is compact.
\end{proof}

\section{Asai cube factors via the Weil$-$Deligne representations}\label{S:Asai cubic Galois}

Let $F$ be a local field of characteristic zero and $\psi$ a non-trivial additive character of $F$. Let $W_F'$ be the Weil$-$Deligne group of $F$. We identify characters of $F^\times$ with one-dimensional admissible representations of $W_F'$ by local class field theory. 

Let $F'$ be a finite extension of degree $d$ over $F$. We identify the Langlands dual group ${}^L({\rm R}_{F'/F}\GL_2)$ of ${\rm R}_{F'/F}\GL_2$ with $\GL_2(\C)^d \rtimes \Gal(\overline{F}/F)$ (cf.\,\cite[\S\,5]{Borel1979}), where the action of $\Gal(\overline{F}/F)$ on $\GL_2(\C)^d$ is the permutation of components induced by the natural homomorphism $\Gal(\overline{F}/F) \rightarrow \Gal(F'/F)$. Let ${\rm As}$ be the Asai representation of ${}^L({\rm R}_{F'/F}\GL_2)$ on $(\C^2)^{\otimes d} = \C^2 \otimes \cdots \otimes \C^2$ so that the restriction of ${\rm As}$ to $\GL_2(\C)^d$ is defined by
\[
{\rm As}(g_1,\cdots,g_d) \cdot (v_1 \otimes \cdots\otimes v_d) = (g_1\cdot v_1,\cdots,g_d \cdot v_d)
\]
and the action of $\Gal(\overline{F}/F)$ on $(\C^2)^{\otimes d}$ is the permutation of components induced by the natural  homomorphism $\Gal(\overline{F}/F) \rightarrow \Gal(F'/F)$.
Let $\pi$ be an irreducible admissible representation of $({\rm R}_{F'/F}\GL_2)(F) = \GL_2(F')$ with central character $\omega_\pi$. Denote by $
\phi_\pi : W_F' \rightarrow ^L({\rm R}_{F'/F}\GL_2)$ the $L$-parameter associated to $\pi$ via the local Langlands correspondence. Then we have a $2^d$-dimensional admissible representation
\[
{\rm As}\circ \phi_\pi : W_F' \longrightarrow \GL((\C^2)^{\otimes d}).
\]

Assume $d=2$ and $\pi = {\rm Ind}_{B(F')}^{\GL_2(F')}(\chi_1\otimes \chi_2)$ for some characters $\chi_1$ and $\chi_2$ of $(F')^\times$. Then
\begin{align}\label{E:Asai quadratic}
{\rm As}\circ \phi_\pi = \chi_1 \vert_{F^\times} \oplus \chi_2 \vert_{F^\times} \oplus {\rm Ind}_{W_{F'}'}^{W_F'}(\chi_1\chi_2^\sigma),
\end{align}
where $\sigma$ is the generator of $\Gal(F'/F)$ and $\chi_2^{\sigma}(a) = \chi_2(\sigma(a))$. 

Assume $d=3$ and $\pi = {\rm Ind}_{B(F')}^{\GL_2(F')}(\chi_1\otimes \chi_2)$ for some characters $\chi_1$ and $\chi_2$ of $(F')^\times$. When $F'$ is Galois over $F$, we have
\begin{align}\label{E:Asai cubic 1}
{\rm As}\circ \phi_\pi = \chi_1 \vert_{F^\times} \oplus \chi_2 \vert_{F^\times} \oplus {\rm Ind}_{W_{F'}'}^{W_F'}(\chi_1\chi_1^\sigma\chi_2^{\sigma^2} \oplus \chi_1\chi_2^\sigma\chi_2^{\sigma^2}),
\end{align}
where $\sigma$ is a generator of $\Gal(F'/F)$ and $\chi_i^{\sigma^j}(a) = \chi_i(\sigma^j(a))$. When $F'$ is not Galois over $F$, we have
\begin{align}\label{E:Asai cubic 2}
{\rm As}\circ \phi_\pi = \chi_1 \vert_{F^\times} \oplus \chi_2 \vert_{F^\times} \oplus {\rm Ind}_{W_{F'}'}^{W_F'}(\chi_1\chi_2^{-1}(\chi_2\circ{\rm N}_{F'/F}) \oplus \chi_2\chi_1^{-1}(\chi_1\circ{\rm N}_{F'/F})).
\end{align}

Let $E$ be a finite \'etale algebra of degree $d$ over $F$. Let $\psi_E$ the additive character of $E$ defined by $\psi_E = \psi \circ {\rm tr}_{E/F}$. Let $\Pi$ be an irreducible admissible representation of $\GL_2(E)$ with central character $\omega_\Pi$. Assume $E= F_1 \times \cdots \times F_r$ for some finite extension $F_i$ of degree $d_i$ over $F$. Then
\[
\Pi = \pi_1 \times \cdots \times \pi_r
\]
for some irreducible generic admissible representation $\pi_i$ of $\GL_2(F_i)$. Let
\[
({\rm As}\circ \phi_{\pi_1}) \otimes \cdots \otimes ({\rm As}\circ \phi_{\pi_r})
\]
be the admissible representation of $W_F'$ obtained by composing $({\rm As}\circ \phi_{\pi_1})\times \cdots \times({\rm As}\circ \phi_{\pi_r})$ with the tensor representation
\[
\GL((\C^2)^{\otimes d_1})\times\cdots\times\GL((\C^2)^{\otimes d_r}) \longrightarrow \GL((\C^2)^{\otimes d}).
\]
We denote by
\[
L(s,{\rm As}\,\Pi),\quad \varepsilon(s,{\rm As}\,\Pi,\psi) 
\]
the $L$-factor and $\varepsilon$-factor associated to the admissible representation $({\rm As}\circ \phi_{\pi_1}) \otimes \cdots \otimes ({\rm As}\circ \phi_{\pi_r})$ defined as in \cite[\S\,3]{Tate1979}. Let
\[
\gamma(s,{\rm As}\,\Pi,\psi) = \varepsilon(s,{\rm As}\,\Pi,\psi) L(s,{\rm As}\,\Pi)L(1-s,{\rm As}\,\Pi^\vee)
\]
be the associated $\gamma$-factor. Since $({\rm As}\circ \phi_{\pi_1}) \otimes \cdots \otimes ({\rm As}\circ \phi_{\pi_r})$ has determinant $(\omega_\Pi \vert_{F^\times})^{2^{d-1}}$ and dimension $2^d$, we have
\begin{align}\label{E:basic properties WD}
\varepsilon(s,{\rm As}\,\Pi,\psi^a) = \omega_\Pi(a)^{2^{d-1}}|a|_F^{2^d(s-1/2)} \varepsilon(s,{\rm As}\,\Pi,\psi)
\end{align}
for all $a \in F^\times$. 

\begin{lemma}\label{L:convergence WD}
Let $E$ be an \'etale cubic algebra over $F$ and $\Pi$ an irreducible generic admissible representation of $\GL_2(E)$. Then the $L$-factor $L(s,{\rm As}\,\Pi)$ has no poles for 
\[
{\rm Re}(s) > - L(\Pi).
\]
\end{lemma}

\begin{proof} 
We assume $F$ is non-archimedean and $E$ is a field. The other cases can be proved in a similar way and we omit it. When $\Pi$ is essentially square-integrable, the representation $\Pi \otimes |\mbox{ }|_E^{-{\rm wt}(\omega_\Pi)/2}$ is square-integrable. Thus the $L$-factor $L(s,{\rm As}\,(\Pi \otimes |\mbox{ }|_E^{-{\rm wt}(\omega_\Pi)/2}))$ has no poles for ${\rm Re}(s)>0$. Since ${\rm As}\circ (\phi_{\Pi} \otimes |\mbox{ }|_E^{-{\rm wt}(\omega_\Pi)/2}) = ({\rm As}\circ \phi_\Pi)\otimes |\mbox{ }|_F^{-3{\rm wt}(\omega_\Pi)/2}$, we deduce that $L(s,{\rm As}\,\Pi)$ has no poles for 
\[
{\rm Re}(s)>-3{\rm wt}(\omega_\Pi)/2=-L(\Pi).
\]
When $\Pi = {\rm Ind}_{B(E)}^{\GL_2(E)}(\chi_1\otimes\chi_2)$, the $L$-factor $L(s,{\rm As}\,\Pi)$ is equal to
\[
L(s,\chi_1 \vert_{F^\times})L(s,\chi_2 \vert_{F^\times})L(s,\chi_1\chi_1^\sigma\chi_2^{\sigma^2})L(s,\chi_1\chi_2^\sigma\chi_2^{\sigma^2})
\]
or
\[
L(s,\chi_1 \vert_{F^\times})L(s,\chi_2 \vert_{F^\times})
L(s,\chi_1\chi_2^{-1}(\chi_2\circ{\rm N}_{F'/F}))L(s,\chi_1^{-1}\chi_2(\chi_1\circ{\rm N}_{F'/F})),
\]
depending on whether $E/F$ is Galois or not. In any case, we see that $L(s,{\rm As}\,\Pi)$ has no poles for 
\begin{align*}
{\rm Re}(s)&>-\min\{ 3{\rm wt}(\chi_1),3{\rm wt}(\chi_2),{\rm wt}(\chi_1)+2{\rm wt}(\chi_2), 2{\rm wt}(\chi_1)+{\rm wt}(\chi_2)\}\\
&=-\min\{ 3{\rm wt}(\chi_1),3{\rm wt}(\chi_2)\}\\
&=-L(\Pi).
\end{align*}
This completes the proof.
\end{proof}

\section{Asai cube factors via the local zeta integrals}\label{S:Asai cubic PSR}
\subsection{Preliminaries}
Recall the similitude symplectic group 
\[
\GSp_6 = \left \{ g \in \GL_{6}\mbox{ } \left \vert \mbox{ } g\bp 0 & {\bf 1}_3 \\ -{\bf 1}_3 & 0 \ep {}^tg = \nu(g) \bp 0 & {\bf 1}_3 \\ -{\bf 1}_3 & 0 \ep ,\mbox{ }\nu(g) \in \mathbb{G}_m    \right .\right \}
\]
and its standard Siegel parabolic subgroup
\[
P= \left \{ \left .  \bp \nu{}^tA^{-1} & * \\ 0 & A\ep \in \GSp_6 \mbox{ } \right\vert \mbox{ }A \in \GL_3, \,\nu \in {\mathbb G}_m\right \}.
\]
Denote by $U$ the unipotent radical of $P$.

Let $F$ be a field of characteristic zero and $E$ an \'etale cubic algebra over $F$. Let $\mathbb K$ be the quadratic discriminant algebra of $E$.
Let
\[
\G = \{g \in {\rm R}_{E/F}\GL_{2} \, \vert\, \det(g) \in \mathbb{G}_m\}
\]
be a linear algebraic group over $F$. Let $(V,\<\, ,\,\>)$ be the nondegenerate symplectic form over $F$ defined by  
\[
V=({\rm R}_{E/F}{\mathbb G}_a)^2,\quad\<x,y\> = {\rm tr}_{E/F}(x_1y_2-x_2y_1)
\]
for $x = (x_1,x_2) , y= (y_1,y_2) \in V$. 
Let
\[
\GSp(V) = \{g \in {\rm R}_{E/F}\GL_{2} \, \vert\, \<xg,yg\> = \nu(g)\<x,y\> \mbox{ for all }x,y \in V,\, \nu(g) \in \mathbb{G}_m\}
\]
be the similitude symplectic group associated to $\<\, ,\,\>$. Then it is easy to verify that $\G$ is a subgroup of $\GSp(V)$ and $\det(g) = \nu(g)$ for $g \in \G$.

Let $X$ and $X_0$ be two maximal isotropic subspaces of $V$ defined by
\[
X = \{(0,y) \in V \,\vert\, y \in {\rm R}_{E/F}{\mathbb G}_a\},\quad X_0 = \{ (x,y) \in V \, \vert\, x \in \mathbb{G}_a,\, {\rm tr}_{E/F}(y)=0 \}.
\]
Define an isomorphism between $X(F)$ and $X_0(F)$ by
\begin{align}\label{E:iso}
\begin{split}
X(F) \longrightarrow X_0(F),\quad (0,x)\longmapsto (3x,0),\quad
(0,y)\longmapsto (0,y)
\end{split}
\end{align}
for $x \in F$ and $y \in E$ with ${\rm tr}_{E/F}(y)=0$. Fix $\eta \in \GSp(V)(F)$ such that 
\begin{align}\label{E:eta}
\begin{split}
&X(F)\cdot\eta=X_0(F),\\
&\nu(\eta) = 1\mbox{ and }\det(\eta \vert_{X(F)})=1 \mbox{ with respect to the isomorphism (\ref{E:iso})}.
\end{split}
\end{align}
We denote by $P_0$ and $R_0$ the stabilizers of $X_0$ in $\GSp(V)$ and $\G$, respectively. Let $U_0$ be the unipotent radical of $R_0$. Note that
\begin{align*}
R_0 &= \{ {\bf a}(t_1){\bf d}(t_2){\bf n}(x) \, \vert \, t_1,t_2 \in \mathbb{G}_m,\, x \in {\rm R}_{E/F}\mathbb{G}_a,\, {\rm tr}_{E/F}(x)=0\},\\
U_0 &=  \{ {\bf n}(x) \, \vert \, x \in {\rm R}_{E/F}\mathbb{G}_a ,\, {\rm tr}_{E/F}(x)=0\}. 
\end{align*}

A symplectic basis of $V(F)$ is an ordered basis $\alpha = \{e_1^*,e_2^*,e_3^*,e_1,e_2,e_3\}$ of $V(F)$ such that
\[
\<e_i^*,e_j\>=\delta_{ij}
\] 
and $\{e_1,e_2,e_3\}$ is an ordered basis of $X(F)$. We write $\alpha_{X} = \{e_1,e_2,e_3\}$.
With respect to a symplectic basis $\alpha$ of $V(F)$, we identify $V(F)$ with the space of row vectors $F^6$. The identification induces an isomorphism 
\begin{align}\label{E:iota}
\iota_\alpha : \GSp(V)(F) \longrightarrow \GSp_6(F).
\end{align}
Let $\beta$ be another symplectic basis of $V(F)$ and $A_{\alpha,\beta} \in \GL_3(F)$ be the transition matrix from $\alpha_{X}$ to $\beta_{X}$. We recall that the transition matrix $A_{\alpha,\beta} = (a_{ij})_{1\leq i,j \leq 3}$ is defined so that 
\[
e_i' = a_{i1}\cdot e_1 + a_{i2}\cdot e_2 + a_{i3}\cdot e_3,
\]
where $\beta_{X} = \{e_1',e_2',e_3'\}$.
Then we have
\begin{align}\label{E:transition}
\iota_\beta(g) = p_{\alpha,\beta}\cdot \iota_\alpha(g)\cdot p_{\alpha,\beta}^{-1}
\end{align}
for $g \in \GSp(V)(F)$, where $p_{\alpha,\beta} \in P$ is of the form
\[
p_{\alpha,\beta} = \bp {}^tA_{\alpha,\beta}^{-1}  & * \\ 0 &A_{\alpha,\beta}\ep.
\]

\subsection{Local zeta integrals and local factors}\label{SS:local factor}
Let $F$ be a local field of characteristic zero and $\psi$ be a non-trivial additive character of $F$. 
Let
\[
K = \left \{ \begin{array}{lll} \GSp_6(F)\cap\GL_{6}(\o_F)& \mbox{ if $F$ is non-archimedean},\\
\GSp_6(\R) \cap {\rm O}(6) & \mbox{ if $F=\R$},\\
\GSp_6(\C) \cap {\rm U}(6) & \mbox{ if $F=\C$},
\end{array} \right .
\]
be a maximal compact subgroup of $\GSp_6(F)$.

Let $\omega$ be a character of $F^\times$. For $s \in \C$, let $\chi_{\omega,s}$ be the character of $P(F)$ defined by
\[
\chi_{\omega,s}(p) = \omega(\nu\cdot \det(A)^{-1}) \cdot \delta_{P}(p)^{s/2-1/4}
\]
for $p = \bp \nu{}^tA^{-1} & * \\ 0 & A\ep$, where $\delta_{P}$ is the modulus character of $P$. Recall that $\delta_P(p) = |\nu^6\det(A)^{-4}|_F$.
Denote by $I(\omega,s)$ the degenerate principal series representation
\[
{\rm Ind}_{P(F)}^{\GSp_6(F)}(\chi_{\omega,s})
\]
and by $\rho$ the right translation action of $\GSp_6(F)$ on $I(\omega,s)$. Recall that $I(\omega,s)$ consisting of smooth functions $f : \GSp_6(F) \rightarrow \C$ such that 
\begin{itemize}
\item For $p \in P(F)$ and $g \in \GSp_6(F)$, 
\[
f(pg) = \chi_{\omega,s}(p)\delta_{P}(p)^{1/2}f(g).
\]
\item $f$ is right $K$-finite.
\end{itemize}
Let ${\rm w} \in \GSp_6(F)$ be the Weyl element defined by
\[
{\rm w} = \bp 0 & -J \\ J & 0\ep,
\]
where $J \in \GL_3$ is the anti-diagonal matrix with non-zero entries all equal to $1$. We define the intertwining operator 
\begin{align*}
M_{\rm w}(\omega)&:I(\omega,s) \longrightarrow I(\omega^{-1},1-s),\\
M_{\rm w}(\omega) f(g) &= \omega(\nu(g)) \int_{U(F)}f({\rm w}^{-1}ug)\,du.
\end{align*}
The integral is absolutely convergent for ${\rm Re}(s)$ sufficiently large and can be meromorphically continued to $s \in \C$. Let $M_{{\rm w},\psi}^*(\omega):I(\omega,s) \longrightarrow I(\omega^{-1},1-s)$ be the normalized intertwining operator defined by
\begin{align}\label{E:io}
M_{{\rm w},\psi}^*(\omega) = \gamma(2s-2,\omega,\psi)\gamma(4s-3,\omega^2,\psi) M_{\rm w}(\omega).
\end{align}
Then $M_{{\rm w},\psi}^*(\omega^{-1})\circ M_{{\rm w},\psi}^*(\omega)$ is a scalar multiple of the identity map on $I(\omega,s)$. We normalize the Haar measure on $U(F)$ so that 
\[
M_{{\rm w},\psi}^*(\omega^{-1})\circ M_{{\rm w},\psi}^*(\omega) = {\rm id}.
\]
We write $M_{{\rm w},\psi}^*(\omega) = M_{\rm w}^*$ if there is no cause of confusion. 
A map 
\[
\C \times \GSp_6(F) \longrightarrow \C, \quad (s,g) \longmapsto f_s(g)
\]
is a holomorphic section of $I(\omega,s)$ if it satisfies the following conditions:
\begin{itemize}
\item For each $s \in \C$, the function $g \mapsto f_s(g)$ belongs to $I(\omega,s)$.
\item For each $g \in \GSp_6(F)$, the function $s \mapsto f_s(g)$ is holomorphic.
\item $f_s$ is right $K$-finite.
\end{itemize}
A map 
\[
\C \times \GSp_6(F) \longrightarrow \C, \quad (s,g) \longmapsto f_s(g)
\]
is a good section of $I(\omega,s)$ if it satisfies the following conditions:
\begin{itemize}
\item The map $(s,g)\mapsto L(2s+1,\omega)^{-1}L(4s,\omega^2)^{-1}f_s(g)$ is a holomorphic section of $I(\omega,s)$.
\item The map $(s,g)\mapsto L(3-2s,\omega^{-1})^{-1}L(4-4s,\omega^{-2})^{-1}M_{\rm w}^*f_s(g)$ is a holomorphic section of $I(\omega^{-1},1-s)$.
\end{itemize}
By \cite[Lemma 1.3]{Ikeda1992}, every holomorphic section is a good section.

Let $\Pi$ be an irreducible generic admissible representation of $\GL_2(E)$ with central character $\omega_{\Pi}$. Write $\omega = \omega_{\Pi}\vert_{F^\times}$. Recall $\mathcal{W}(\Pi,\psi_E)$ is the space of Whittaker functions of $\Pi$ with respect to $\psi_E = \psi\circ{\rm tr}_{E/F}$. Let $f_s$ be a good section of $I(\omega,s)$ and $W \in \mathcal{W}(\Pi,\psi_E)$. Define the local zeta integral 
\[
Z_\alpha(f_s,W) = \int_{F^\times U_0(F) \backslash \G(F)}f_s(\iota_\alpha(\eta g))W(g)\,dg,
\]
where $\alpha$ is a symplectic basis of $V(F)$, $\iota_\alpha$ is the isomorphism (\ref{E:iota}), and $\eta \in \GSp(V)(F)$ satisfies (\ref{E:eta}). Note that the integral is independent of the choice of $\eta$. Indeed, if $\eta' \in \GSp(V)(F)$ also satisfies (\ref{E:eta}), then $\iota_\alpha(\eta'\eta^{-1}) \in P(F)$ and $\nu(\eta'\eta^{-1}) = \det(\eta'\vert_{X(F)}\eta\vert_{X(F)}^{-1}) =1$. 
By the results in \cite{PSR1987} and \cite{Ikeda1989}, the integral is absolutely convergent for ${\rm Re}(s)$ sufficiently large and admits meromorphic continuation to $s \in \C$.
Moreover, by \cite[Proposition 3.1]{PSR1987} and \cite[Proposition 4.2]{Ikeda1989}, there exists a unique meromorphic function $\gamma_{\rm PSR}(s,{\rm As}\,\Pi,\psi,\alpha)$, called the $\gamma$-factor, such that we have the functional equation
\begin{align}\label{E:fe}
Z_\alpha(M_{\rm w}^*f_s,W^\vee) = \gamma_{\rm PSR}(s,{\rm As}\,\Pi,\psi,\alpha) Z_\alpha(f_s,W).
\end{align}
Recall that $W^\vee \in \mathcal{W}(\Pi^\vee,\psi_E)$ is defined by
\[
W^\vee(g) = \omega_\Pi(\det(g))^{-1}W(g).
\]
By Lemma \ref{L:basic properties zeta}-(1) below, we see that $\gamma_{\rm PSR}(s,{\rm As}\,\Pi,\psi,\alpha) = \gamma_{\rm PSR}(s,{\rm As}\,\Pi,\psi,\beta)$ if $\alpha_X = \beta_X$. Therefore we also write $\gamma_{\rm PSR}(s,{\rm As}\,\Pi,\psi,\alpha_X) = \gamma_{\rm PSR}(s,{\rm As}\,\Pi,\psi,\alpha)$. In particular, identifying $E$ with $X(F)$ via the isomorphism
\[
E \longrightarrow X(F),\quad x \longmapsto (0,x),
\]
the notion $\gamma_{\rm PSR}(s,{\rm As}\,\Pi,\psi,\alpha)$ is defined for any basis $\alpha$ of $E$ over $F$.
\begin{lemma}\label{L:basic properties zeta}

(1) Let $\alpha, \beta$ be symplectic bases of $V(F)$. We have
\[
\gamma_{\rm PSR}(s,{\rm As}\,\Pi,\psi,\beta) = \omega(\det(A_{\alpha,\beta}))^{2}|\det(A_{\alpha,\beta})|_F^{4s-2}\gamma_{\rm PSR}(s,{\rm As}\,\Pi,\psi,\alpha).
\]
Here $A_{\alpha,\beta} \in \GL_3(F)$ is the transition matrix from $\alpha_{X}$ to $\beta_{X}$.

(2) Let $a \in F^\times$ and $\alpha$ be a symplectic basis of $V(F)$. We have
\[
\gamma_{\rm PSR}(s,{\rm As}\,\Pi,\psi^a,\alpha) = \omega(a)^4 |a|_F^{8s-4}\gamma_{\rm PSR}(s,{\rm As}\,\Pi,\psi,\alpha).
\]
\end{lemma}

\begin{proof}
Let $f_s$ be a good section of $I(\omega,s)$ and $W \in \mathcal{W}(\Pi,\psi_E)$. First we prove assertion (1). By (\ref{E:transition}), 
\begin{align}\label{E:transition of integrals}
\begin{split}
Z_\beta(f_s,W) &= \int_{F^\times U_0(F) \backslash \G(F)}f_s(p_{\alpha,\beta}\cdot \iota_\alpha(\eta g)\cdot p_{\alpha,\beta}^{-1})W(g)\,dg\\
& = \omega(\det(A_{\alpha,\beta}))^{-1}|\det(A_{\alpha,\beta})|_F^{-2s-1} Z_\alpha(\rho(p_{\alpha,\beta}^{-1})f_s,W).
\end{split}
\end{align}
Similarly, we have
\[
Z_\beta(M_{\rm w}^*f_s,W^\vee) = \omega(\det(A_{\alpha,\beta}))|\det(A_{\alpha,\beta})|_F^{2s-3} Z_\alpha(\rho(p_{\alpha,\beta}^{-1})M_{\rm w}^*f_s,W^\vee).
\]
Assertion (1) then follows from the functional equation (\ref{E:fe}).

To prove assertion (2), it follows from assertion (1) that we may assume 
\[
\alpha = \{(1,0), (\delta_1^*,0), (\delta_2^*,0),(0,1/3),(0,\delta_1),(0,\delta_2)\}
\]
for some $\delta_1, \delta_1^*, \delta_2, \delta_2^* \in E$ with ${\rm tr}_{E/F}(\delta_1)={\rm tr}_{E/F}(\delta_2)={\rm tr}_{E/F}(\delta_1^*)={\rm tr}_{E/F}(\delta_2^*)=0$ and $\eta \in \GSp(V)(F)$ satisfying (\ref{E:eta}) is defined by
\[
(0,1/3)\cdot \eta = (1,0),\quad (\delta_i^*,0) \cdot \eta = (\delta_i^*,0),\quad (0,\delta_i) \cdot \eta = (0,\delta_i).
\] Then it is easy to verify that $\iota_\alpha(\eta{\bf a}(a^{-1})\eta^{-1}) = {\rm diag}(1,a^{-1},a^{-1},a^{-1},1,1)$.
Let $W' \in \mathcal{W}(\Pi,\psi_E^a)$ defined by $W'(g) = W({\bf a}(a)g)$. We write $Z_{\alpha,\psi}(f_s,W) = Z_{\alpha}(f_s,W)$ and $M_{{\rm w} ,\psi} = M_{\rm w}$ to emphasis the dependence on $\psi$. We have
\begin{align*}
Z_{\alpha,\psi^a}(f_s,W') &= \int_{F^\times U_0(F) \backslash \G(F)}f_s(\iota_\alpha(\eta g))W({\bf a}(a)g)\,dg\\
&= |a|_F \int_{F^\times U_0(F) \backslash \G(F)}f_s(\iota_\alpha(\eta {\bf a}(a^{-1})g))W(g)\,dg\\
&= |a|_F^{-s+1/2}Z_{\alpha,\psi}(f_s,W).
\end{align*}
Note that
\begin{align*}
\gamma(2s-2,\omega,\psi^a)\gamma(4s-3,\omega^2,\psi^a) &= \omega(a)^3 |a|_F^{6s-6}\gamma(2s-2,\omega,\psi)\gamma(4s-3,\omega^2,\psi),\\
M_{{\rm w},\psi^a} &= |a|_F^{3}M_{{\rm w},\psi}, \\
(W')^\vee(g) &= \omega(a) W^\vee({\bf a}(a)g).
\end{align*}
Similarly, we have
\begin{align*}
Z_{\alpha,\psi^a}(M_{{\rm w},\psi^a}^*f_s,(W')^\vee) &= \omega(a)^4 |a|_F^{6s-3} \int_{F^\times U_0(F) \backslash \G(F)}M_{{\rm w},\psi}^*f_s(\iota_\alpha(\eta g))W^\vee({\bf a}(a)g)\,dg\\
&=\omega(a)^4 |a|_F^{7s-7/2}Z_{\alpha,\psi}(M_{{\rm w},\psi}^*f_s,W^\vee).
\end{align*}
Assertion (2) then follows from the functional equation (\ref{E:fe}).
\end{proof}

Assume $F$ is non-archimedean. By \cite[Appendix 3 to \S\,3]{PSR1987}, the $\C[q_F^s,q_F^{-s}]$-module generated by $Z_\alpha(f_s,W)$ for good sections $f_s$ of $I(\omega,s)$ and $W \in \mathcal{W}(\Pi,\psi_E)$ is a fractional ideal of $\C[q_F^s,q_F^{-s}]$ containing $1$. Therefore, there is a unique generator of the form $P(q_F^{-s})$ with $P(X) \in \C[X]$ and $P(0)=1$. Define the $L$-factor and $\varepsilon$-factor as follows: 
\begin{align*}
L_{\rm PSR}(s,{\rm As}\,\Pi)&=P(q^{-s})^{-1},\\
\varepsilon_{\rm PSR}(s,{\rm As}\,\Pi,\psi,\alpha)&=\gamma_{\rm PSR}(s,{\rm As}\,\Pi,\psi,\alpha)L_{\rm PSR}(s,{\rm As}\,\Pi)L_{\rm PSR}(1-s,{\rm As}\,\Pi^{\vee})^{-1}.
\end{align*}
Note that $\varepsilon$-factor is a unit in $\C[q_F^{s},q_F^{-s}]$ by the functional equation (\ref{E:fe}).

Assume $F$ is archimedean. Up to holomorphic functions without zeros, there exists a unique meromorphic function $L_{\rm PRS}(s,{\rm As}\,\Pi)$ without zeros, called the $L$-factor, satisfying the following conditions:
\begin{itemize}
\item $L_{\rm PRS}(s,{\rm As}\,\Pi)^{-1}Z(f_s,W)$ is holomorphic for any good section $f_s$ of $I(\omega,s)$ and $W \in \mathcal{W}(\Pi,\psi_E)$.
\item For each $s_0 \in \C$, there exist a good section $f_s$ and $W \in \mathcal{W}(\Pi,\psi_E)$ such that $L_{\rm PRS}(s,{\rm As}\,\Pi)^{-1}Z(f_s,W)$ is non-zero at $s=s_0.$
\end{itemize}
Define the $\varepsilon$-factor 
\begin{align*}
\varepsilon_{\rm PSR}(s,{\rm As}\,\Pi,\psi,\alpha)&=\gamma_{\rm PSR}(s,{\rm As}\,\Pi,\psi,\alpha)L_{\rm PSR}(s,{\rm As}\,\Pi)L_{\rm PSR}(1-s,{\rm As}\,\Pi^{\vee})^{-1},
\end{align*}
which is well-defined up to holomorphic function without zeros. By the properties characterizing the $L$-factor above and the functional equation (\ref{E:fe}), the $\varepsilon$-factor is a holomorphic function without zeros.

Recall the domain $\mathcal{D}(\Pi)$ associated to $\Pi$ defined in (\ref{E:domain}). Let $W_\lambda$ be a holomorphic family of Whittaker functions of $\Pi_\lambda$ with respect to $\psi_E$. Write $\omega_\lambda = \omega_{\Pi_\lambda} \vert_{F^\times}$. Similar to the case for $I(\omega,s)$, we define the notion of holomorphic sections and good sections of $I(\omega_\lambda,s)$ for $(s,\lambda)$ varying in $\C \times \mathcal{D}(\Pi)$.

The following lemma is a variant of the estimation in \cite[Proposition 3.2]{PSR1987} and \cite[\S\,3.1]{Ikeda1989} by replacing a single Whittaker function with a holomorphic family of Whittaker functions.

\begin{lemma}\label{L:convergence PSR}
Assume $F$ is non-archimedean. 
Let $W_\lambda$ be a holomorphic family of Whittaker functions of $\Pi_\lambda$ with respect to $\psi_E$. Write $\omega_\lambda = \omega_{\Pi_\lambda} \vert_{F^\times}$. Let $f_{s,\lambda}$ be a good section of $I(\omega_\lambda,s)$. Then the integral $Z_\alpha(f_{s,\lambda},W_\lambda)$ is absolutely convergent for 
\[
{\rm Re}(s) > -L(\Pi_\lambda),
\]
uniformly for $s$ and $\lambda$ varying in compact sets. In particular, the integral $Z_\alpha(f_{s,\lambda},W_\lambda)$ defines a holomorphic function on the domain
\[
\{(s,\lambda) \in \C \times \mathcal{D}(\Pi) \,\vert\, {\rm Re}(s) > -L(\Pi_\lambda)\}.
\]
Here $L(\Pi_\lambda) \in \R$ is defined as in (\ref{E:parameter}).
\end{lemma}

\begin{proof}
We have the following three cases:
\[
\begin{cases}
E=F\times F \times F & \mbox{ Case 1},\\
E=F'\times F \mbox{ for some quadratic extension $F'$ of $F$}& \mbox{ Case 2},\\
E \mbox{ is a field} & \mbox{ Case 3}.
\end{cases}
\]
Then $\Pi = \pi_1 \times \pi_2 \times \pi_3$ for some irreducible generic admissible representations $\pi_i$ of $\GL_2(F)$ in Case 1, and $\Pi = \pi_1 \times \pi_2$ for some irreducible generic admissible representations $\pi_1$ and $\pi_2$ of $\GL_2(F')$ and $\GL_2(F)$, respectively, in Case 2. In Case 3, we write $\Pi = \pi_1$.
Let
\[
r = \begin{cases}
3 & \mbox{ Case 1},\\
2 & \mbox{ Case 2},\\
1 & \mbox{ Case 3},
\end{cases},\quad
E^\times = \begin{cases}
F^\times \times F^\times \times F^\times & \mbox{ Case 1},\\
(F')^\times \times F^\times & \mbox{ Case 2},\\
E^\times & \mbox{ Case 3}.
\end{cases}
\]
Define $(d_1,\cdots,d_r) \in \Z^r$ by
\[
(d_1,\cdots,d_r) = \begin{cases}
(1,1,1) & \mbox{ Case 1},\\
(2,1)   & \mbox{ Case 2},\\
3       & \mbox{ Case 3}.
\end{cases}
\]
We write $\lambda = (\lambda_1,\cdots,\lambda_r) \in \mathcal{D}(\pi_1)\times \cdots \times \mathcal{D}(\pi_r) = \mathcal{D}(\Pi)$. Let $\mathcal{C}$ be a complete set of coset representatives for $E^\times / (F^\times)^r$. Let $s_\lambda = {\rm wt}(\omega_\lambda)$ and $s_{i,\lambda} = {\rm wt}(\omega_{({\pi_i})_{\lambda_i}} \vert_{F^\times})$ for $1\leq i \leq r$. Then $s_\lambda = \sum_{i=1}^r s_{i,\lambda}$. Note that 
\[
L(\Pi_\lambda) \leq s_\lambda/2
\]
by definition.

Let ${\bf K} = \G(F) \cap \GL_2(\o_E)$. Let $f_{s,\lambda}^o$ be the $K$-invariant good section of $I(|\omega_\lambda|,s)$ normalized so that $f_{s,\lambda}^o(1)=1$.
Since $E^\times / (F^\times)^r$ is compact, by the $K$-finiteness of $f_{s,\lambda}$, there exists a constant $C_{s,\lambda}>0$ bounded uniformly as $s$ and $\lambda$ vary in a compact set such that 
\begin{align}\label{E:inequality 1}
|f_{s,\lambda}(g\cdot\iota_\alpha({\bf m}(t)k))| \leq C_{s,\lambda} \cdot |f_{s,\lambda}^o(g)|
\end{align}
for $g \in \GSp_6(F)$, $t \in \mathcal{C}$, and $k \in {\bf K}$. Let $\epsilon>0$. By Lemma \ref{L:analytic family 2}, there exist an integer $n$ independent of $\epsilon$, $\lambda$ and a constant $C_{\lambda,\epsilon}>0$ bounded uniformly as $\lambda$ varies in a compact set such that
\begin{align}\label{E:inequality 2}
|W_\lambda({\bf a}(\nu){\bf m}(t)k)| \leq C_{\lambda,\epsilon}\cdot \varphi(\nu)\prod_{i=1}^r|\nu_i|_F^{d_i\cdot l((\pi_i)_{\lambda_i})+d_i/2-\epsilon/r}
\end{align}
for $\nu = (\nu_1,\cdots,\nu_r)\in (F^\times)^r$, $t \in \mathcal{C}$, and $k \in {\bf K}$, where $\varphi \in \mathcal{S}(F^r)$ is given by 
\[
\varphi((\nu_1,\cdots,\nu_r)) = \prod_{i=1}^r\mathbb{I}_{\varpi_F^{-n}\o_F}(\nu_i).
\]
We have
\[
Z_\alpha(f_{s,\lambda},W_\lambda) = Z^{(0)}_\alpha(f_{s,\lambda},W_\lambda) + q_F^2 \cdot Z^{(1)}_\alpha(f_{s,\lambda},W_\lambda),
\]
where
\begin{align*}
Z^{(0)}_\alpha(f_{s,\lambda},W_\lambda) &= \int_{F}\int_{E^\times}\int_{\bf K} f_{s,\lambda}(\iota_\alpha(\eta{\bf u}(x){\bf m}(t)k))W_\lambda({\bf u}(x){\bf m}(t)k)|t|_E^{-2}\,dk\,d^\times t\,dx,\\
Z^{(1)}_\alpha(f_{s,\lambda},W_\lambda) &= \int_{F}\int_{E^\times}\int_{\bf K} f_{s,\lambda}(\iota_\alpha(\eta{\bf a}(\varpi_F){\bf u}(x){\bf m}(t)k))W_\lambda({\bf a}(\varpi_F){\bf u}(x){\bf m}(t)k)|t|_E^{-2}\,dk\,d^\times t\,dx.
\end{align*}
Here ${\bf u}(x) \in \G(F)$ is defined by
\[
{\bf u}(x) = \begin{cases}
(1,1,{\bf n}(x)) & \mbox{ Case 1},\\
(1,{\bf n}(x))& \mbox{ Case 2},\\
{\bf n}(x/3)& \mbox{ Case 3}.
\end{cases}
\]
By (\ref{E:transition of integrals}), without lose of generality, we assume $\alpha = \{e_1^*,e_2^*,e_3^*,e_1,e_2,e_3\}$ is given as follows:
\begin{itemize}
\item In Case 1,
\begin{align*}
e_1^* &= \left( (1,0,0),(0,0,0)\right),\quad e_2^* = \left( (0,1,0),(0,0,0)\right),\quad e_3^* = \left( (0,0,1),(0,0,0)\right),\\
e_1 &= \left( (0,0,0),(1,0,0)\right),\quad e_2 = \left( (0,0,0),(0,1,0)\right),\quad e_3 = \left( (0,0,0),(0,0,1)\right).
\end{align*}
\item In Case 2, 
\begin{align*}
e_1^* &= \left((1,0) ,(0,0)\right),\quad e_2^* = \left( (\delta^*,0), (0,0)\right) ,\quad e_3^* = \left((0,1),(0,0)\right),\\
e_1 &= \left((0,0) ,(1/2,0)\right),\quad e_2 = \left( (0,0), (\delta,0)\right) ,\quad e_3 = \left((0,0),\left(0,1\right)\right)
\end{align*}
for some $\delta, \delta^* \in F'$ with ${\rm tr}_{F'/F}(\delta)={\rm tr}_{F'/F}(\delta^*)=0$. 
\item In Case 3,
\[
e_1^* = \left(1,0\right), \quad e_2^* = (\delta_1^*,0), \quad e_3^* = (\delta_2^*,0),\quad e_1 = (0,1/3),\quad e_2 = (0,\delta_1), \quad e_3 = (0,\delta_2)
\] 
for some $\delta_1, \delta_1^*, \delta_2, \delta_2^* \in E$ with ${\rm tr}_{E/F}(\delta_1)={\rm tr}_{E/F}(\delta_2)={\rm tr}_{E/F}(\delta_1^*)={\rm tr}_{E/F}(\delta_2^*)=0$.
\end{itemize}
We define $\eta \in \GSp(V)(F)$ as follows:
\begin{itemize}
\item We have
\[
e_1^*\cdot \eta = -e_1,\quad e_2^* \cdot \eta = e_2^*,\quad e_3^*\cdot \eta = e_3^*.
\]
\item In Case 1,
\[
e_1\cdot \eta  = e_1^*+e_2^*+e_3^*,\quad e_2 \cdot \eta = -e_1+e_2,\quad e_3 \cdot \eta = -e_1+e_3.
\]
\item In Case 2,
\[
e_1\cdot \eta  = e_1^*+e_3^*,\quad e_2\cdot \eta = e_2,\quad e_3\cdot \eta = -e_1+e_3.
\]
\item In Case 3,
\[
e_1\cdot\eta = e_1^*,\quad e_2 \cdot \eta= e_2,\quad e_3\cdot\eta = e_3.
\]
\end{itemize}
It is easy to verify that $\eta$ satisfies (\ref{E:eta}).
Then
\[
\iota_\alpha(\eta{\bf u}(x){\bf m}(t)) = \left\{\begin{array}{cl}
\bp  *&*&*&*&*&* \\ *&*&*&*&*&* \\ *&*&*&*&*&* \\ t_1&t_2&t_3&t_1^{-1}x&t_2^{-1}x&t_3^{-1}x \\ 0&0&0&-t_1^{-1}&t_2^{-1}&0 \\ 0&0&0&-t_1^{-1}&0&t_3^{-1} \ep & \mbox{ Case 1},\\
\bp *&*&*&*&*&* \\ *&*&*&*&*&* \\ *&*&*&*&*&* \\ t_1&0&t_2&0&0&t_2^{-1}x\\0&0&0&0&t_1^{-1}&0\\0&0&0&-t_1^{-1}&0&t_2^{-1} \ep & \mbox{ Case 2},\\
\bp *&*&*&*&*&* \\ *&*&*&*&*&* \\ *&*&*&*&*&* \\ t_1&0&0&t_1^{-1}x&0&0 \\ 0&0&0&0&t_1^{-1}&0 \\ 0&0&0&0&0&t_1^{-1} \ep & \mbox{ Case 3},
\end{array} \right .
\]
for $x \in F$ and $t = (t_1,\cdots,t_r) \in (F^\times)^r$. Therefore,
\begin{align}\label{E:inequality 3}
f_{s,\lambda}^o(\iota_\alpha(\eta{\bf u}(x){\bf m}(t))) = \left(\max \{|x|_F,|t_1|_F^2,\cdots,|t_r|_F^2\}\prod_{i=1}^r|t_i|_F^{-d_i}\right)^{-2s-s_\lambda-1} 
\end{align}
for $x \in F$ and $t = (t_1,\cdots,t_r) \in (F^\times)^r$. 
We deduce from (\ref{E:inequality 1})-(\ref{E:inequality 3}) that the integral $Z^{(0)}_\alpha(f_{s,\lambda},W_\lambda)$ is majorized by
\begin{align*}
&C_{s,\lambda} C_{\lambda,\epsilon} \int_F \max\{|x|_F,1\}^{-2{\rm Re}(s)-s_\lambda-1}\,dx \\
&\times \int_{(F^\times)^r}\varphi(t^2)\prod_{i=1}^r|t_i|_F^{2d_i\cdot{\rm Re}(s)+2d_i\cdot l((\pi_i)_{\lambda_i})+d_i\cdot s_\lambda-s_{i,\lambda}-2\epsilon/r}\max\{|t_1|_F^2, \cdots,|t_r|_F^2\}^{-2{\rm Re}(s)-s_\lambda}\,d^\times t.
\end{align*}
The above integral is absolutely convergent for
\[
{\rm Re}(s) > \max\{-s_\lambda/2, -L(\Pi_\lambda)+\epsilon\} = -L(\Pi_\lambda)+\epsilon.
\]
Moreover, it is clear that the above integral is uniformly convergent as $s$ and $\lambda$ vary in compact sets. Therefore we obtain a uniform estimate for $Z^{(0)}_\alpha(f_{s,\lambda},W_\lambda)$. We have a similar estimate for $Z^{(1)}_\alpha(f_{s,\lambda},W_\lambda)$. This completes the proof.
\end{proof}

In the following lemma, we recall the known cases of Theorem \ref{T:1} in the literature.
\begin{lemma}\label{L:unramified gamma}
Assume one of the following assumptions is satisfied:
\begin{itemize}
\item $F$ is archimedean.
\item $E=F\times F \times F$.
\item $F$ is non-archimedean, $E$ is unramified over $F$, and $\Pi$ is unramified.
\item $E$ is not a field and $\Pi$ is unramified.
\end{itemize}
We have
\[
\gamma_{\rm PSR}(s,{\rm As}\,\Pi,\psi,\alpha) = \omega_\Pi(\Delta_{E/F}(\alpha))|\Delta_{E/F}(\alpha)|_F^{2s-1}\omega_{\mathbb{K}/F}(-1)\gamma(s,{\rm As}\,\Pi,\psi)
\]
for any basis $\alpha$ of $E$ over $F$ and non-trivial additive character $\psi$ of $F$. Here $\Delta_{E/F}(\alpha)$ is the relative discriminant of $\alpha$ for $E/F$ and $\omega_{\mathbb{K}/F}$ is the quadratic character associated to $\mathbb{K}/F$ by local class field theory.
\end{lemma}

\begin{proof}
We assume that $F$ is non-archimedean, $E$ is unramified over $F$, and $\Pi$ is unramified. The rest of the cases follow from \cite[Theorem 3]{Ikeda1989}, \cite[Theorem 4.4.1]{Rama2000}, and \cite[Theorem B]{CCI2018}. By (\ref{E:basic properties WD}) and Lemma \ref{L:basic properties zeta}, we may further assume that $\alpha= \{x_1,x_2,x_3\}$ is an integral basis of $\o_E$ over $\o_F$ and $\psi$ is of conductor $\o_F$. Note that in this case $\omega_\Pi(\Delta_{E/F}(\alpha))|\Delta_{E/F}(\alpha)|_F^{2s-1}\omega_{\mathbb{K}/F}(-1)=1$ by assumption. Let $f_s^o$ be the $K$-invariant good section of $I(\omega,s)$ and $W^o \in \mathcal{W}(\Pi,\psi)$ the $\GL_2(\o_E)$-invariant Whittaker function normalized so that $f_s^o(1)=1$ and $W^o(1)=1$. Let $\alpha^* = \{x_1^*,x_2^*,x_3^*\}$ be the dual basis of $\alpha$ and $\tilde{\alpha}$ the symplectic basis of $V(F)$ defined by
\[
\tilde{\alpha} = \{(x_1^*,0),(x_2^*,0),(x_3^*,0),(0,x_1),(0,x_2),(0,x_3)\}.
\]
By \cite[Theorem 3.1]{PSR1987}, we have
\begin{align*}
Z_{\tilde{\alpha}}(f_{s}^o,W^o) &= L(2s+1,\omega)^{-1}L(4s,\omega^2)^{-1} L(s,{\rm As}\,\Pi),\\
Z_{\tilde{\alpha}}(M_{\rm w}^*f_{s}^o,(W^o)^\vee) &= L(2s+1,\omega)^{-1}L(4s,\omega^2)^{-1} L(1-s,{\rm As}\,\Pi^\vee).
\end{align*}
Here the measure on $F^\times U_0(F) \backslash \G(F)$ is normalized so that ${\rm vol}(\o_F^\times U_0(\o_F)  \backslash {\bf G}(\o_F)) = 1$. 
It follows from the functional equation (\ref{E:fe}) that
\[
\gamma_{\rm PSR}(s,{\rm As}\,\Pi,\psi,\alpha) = \varepsilon(s,{\rm As}\,\Pi,\psi)^{-1}\gamma(s,{\rm As}\,\Pi,\psi).
\]
Since $\Pi$ is unramified and $\psi$ is of order $\o_F$, we have $\varepsilon(s,{\rm As}\,\Pi,\psi)=1$. This completes the proof.
\end{proof}

\subsection{Unramified calculation} In this section, we prove a case of Theorem \ref{T:1} in Corollary \ref{L:unramified gamma2}. This case will be needed in the proof of general case in \S\,\ref{SS:proof 1} below. Assume $F$ is non-archimedean, $E$ is a field, and there is a uniformizer $\varpi_E$ of $\o_E$ such that $\varpi_E^3 = \varpi_F$ is a uniformizer of $\o_F$. In particular, the last assumption is satisfied if $E$ is totally tamely ramified over $F$. Note that $3\varpi_E^{2}\o_E$ is the relative different ideal for $E/F$. Let $\alpha^o$ be the symplectic basis of $V(F)$ defined by
\[
\alpha^o = \left\{ (1,0),(\varpi_E,0),(\varpi_E^2,0),(0,{1}/{3}), (0,{\varpi_E^{-1}}/{3}),(0,{\varpi_E^{-2}}/{3})  \right\}.
\] 

\begin{lemma}\label{L:unram.}
Let $\Pi$ be an irreducible generic admissible representation of $\GL_2(E)$ with central character $\omega_\Pi$. Write $\omega = \omega_\Pi \vert_{F^\times}$. Assume $\Pi$ is unramified and $\psi$ is of conductor $\o_F$. Let $f_s^o$ be the $K$-invariant good section of $I(\omega,s)$ and $W^o \in \mathcal{W}(\Pi,\psi_E)$ the ${\bf a}(3^{-1}\varpi_E^{-2})\GL_2(\o_E){\bf a}(3\varpi_E^{2})$-invariant Whittaker function normalized so that $f_s^o(1)=1$ and $W^o(1)=1$.
We have
\[
Z_{\alpha^o}(f_s^o,W^o) = L(2s+1,\omega)^{-1}L(4s,\omega^2)^{-1} L(s,{\rm As}\,\Pi).
\]
Here the measure on $F^\times U_0(F) \backslash \G(F)$ is normalized so that
\[
{\rm vol}(\o_F^\times (U_0(F) \cap N(\o_E)) \backslash ({\bf G}(F)\cap{\bf a}(\varpi_E^{-2})\GL_2(\o_E){\bf a}(\varpi_E^{2}))) = 1. 
\]
\end{lemma}

\begin{proof}
We write $\varpi = \varpi_F$ and $q=q_F$ for brevity. Assume 
$\Pi = {\rm Ind}_{B(E)}^{\GL_2(E)}(\chi_1\otimes\chi_2)$
for some unramified characters $\chi_1$ and $\chi_2$ of $E^\times$. Let $\alpha = \chi_1(\varpi_E)$, $\beta = \chi_2(\varpi_E)$, and $\chi_1\chi_2 = |\mbox{ }|_E^{s_0}$ for some $s_0 \in \C$.
We have
\begin{align}\label{E:Whittaker formula}
W^o({\bf a}(\varpi_E^n)) = \begin{cases}
\displaystyle{q^{-n/2}\cdot \frac{\alpha^{n+1}-\beta^{n+1}}{\alpha-\beta}}& \mbox{ if $n \geq 0$},\\
0 & \mbox{ if $n <0$}.
\end{cases}
\end{align}
Also note that (cf.\,\cite[p.\,54]{PSR1987})
\begin{align}\label{E:integral formula}
\begin{split}
&\int_F \max\{ |x|_F, q^{-m}     \}^{-s} \psi(\varpi^n x)\,dx \\
&= \begin{cases}
q^{m(s-1)}(1-q^{-(m+n+1)(s-1)})\zeta_F(s-1)\zeta_F(s)^{-1} & \mbox{ if $m+n \geq 0$},\\
0 & \mbox{ if $m+n<0$}.
\end{cases}
\end{split}
\end{align}
Since $\{1,\varpi_E,\varpi_E^2\}$ is an integral basis of $\o_E$ over $\o_F$, one can easily verify that 
\[
\iota_{\alpha^o}({\bf a}(3^{-1}\varpi_E^{-2})\GL_2(\o_E){\bf a}(3\varpi_E^{2}))\subset K.
\]
Therefore, by the invariance of $f_s^o$ and $W^o$, we have 
\[
Z_{\alpha}(f_s^o,W^o) = Z^{(0)}(s) + q^2\cdot Z^{(1)}(s),
\]
where
\begin{align*}
Z^{(0)}(s) &=\int_{F}\int_{E^\times} f_s^o(\iota_{\alpha^o}(\eta{\bf n}(x/3){\bf m}(t)))W^o({\bf n}(x/3){\bf m}(t))|t|_E^{-2}\,d^\times t\,dx,\\
Z^{(1)}(s) &= \int_{F}\int_{E^\times} f_{s}^o(\iota_{\alpha^o}(\eta{\bf a}(\varpi){\bf n}(x/3){\bf m}(t)))W^o({\bf a}(\varpi){\bf n}(x/3){\bf m}(t))|t|_E^{-2}\,d^\times t\,dx.
\end{align*}
Here the measures are normalized so that ${\rm vol}(\o_F)={\rm vol}( \o_E^\times)=1$.
We assume $\eta \in \GSp(V)(F)$ satisfying (\ref{E:eta}) is defined by
\[
(0,1/3) \cdot \eta = (1,0),\quad (\varpi_E^i,0)\cdot \eta = (\varpi_E^i,0),\quad (0,\varpi_E^{-i})\cdot \eta = (0,\varpi_E^{-i})
\]
for $i=1,2$. 
For $x \in F$, $n \in \Z_{\geq 0}$, and $0 \leq i \leq 2$, we have 
\[
\iota_{\alpha^o}(\eta{\bf n}(x/3){\bf m}(\varpi^n \varpi_E^i))=\left\{\begin{array}{cl}
\bp *&*&*&*&*&* \\ *&*&*&*&*&* \\ *&*&*&*&*&* \\ \varpi^{n}&0&0&\varpi^{-n}x&0&0 \\ 0&0&0&0&\varpi^{-n}&0 \\ 0&0&0&0&0&\varpi^{-n} \ep & \mbox{ if $i=0$},\\
\bp *&*&*&*&*&* \\ *&*&*&*&*&* \\ *&*&*&*&*&* \\ 0&\varpi^n&0&0&\varpi^{-n}x&0 \\ 0&0&0&0&0&\varpi^{-n} \\ 0&0&0&\varpi^{-n-1}&0&0 \ep & \mbox{ if $i=1$},\\
\bp *&*&*&*&*&* \\ *&*&*&*&*&* \\ *&*&*&*&*&* \\ 0&0&\varpi^n&0&0&\varpi^{-n}x \\ 0&0&0&\varpi^{-n-1}&0&0 \\ 0&0&0&0&\varpi^{-n-1}&0 \ep & \mbox{ if $i=2$}.
\end{array} \right.
\]
Therefore, 
\begin{align}\label{E:section formula}
f_s^o(\iota_{\alpha^o}(\eta{\bf n}(x/3){\bf m}(\varpi^n \varpi_E^i))) = (\alpha^3\beta^3q^{-2s-1})^{3m+i}\max\{|x|_F,q^{-2n}\}^{-3s_0-2s-1}.
\end{align}
It follows from (\ref{E:Whittaker formula})-(\ref{E:section formula}) that 
\begin{align*}
Z^{(0)}(s) &= L(2s,\omega)L(2s+1,\omega)^{-1}(\alpha-\beta)^{-1} \\
&\times \sum_{i=0}^{2}\alpha^{2i}\beta^{2i} q^{-2is}\sum_{n=0}^{\infty}q^{-2ns}(\alpha^{6n+2i+1}-\beta^{6n+2i+1})(1-(\alpha^3\beta^3q^{-2s})^{2n+1})\\
&= L(2s,\omega)L(2s+1,\omega)^{-1}(\alpha-\beta)^{-1} \\
&\times \Bigg{[}\frac{1-\alpha^{12}\beta^6 q^{-6s}}{1-\alpha^4\beta^2 q^{-2s}}\cdot \left( \frac{\alpha}{1-\alpha^6 q^{-2s}} - \frac{\alpha^4\beta^3 q^{-2s}}{1-\alpha^{12} \beta^6 q^{-6s}}  \right)\\
&\quad\quad- \frac{1-\alpha^{6}\beta^{12} q^{-6s}}{1-\alpha^2\beta^4 q^{-2s}}\cdot \left( \frac{\beta}{1-\beta^6 q^{-2s}} - \frac{\alpha^3\beta^4 q^{-2s}}{1-\alpha^{6} \beta^{12} q^{-6s}}  \right) \Bigg].
\end{align*}
Note that $f_s^o(\iota_{\alpha^o}({\bf a}(\varpi))g) = q^{-s-1/2}f_s^o(g)$. Similarly,
\begin{align*}
Z^{(1)}(s) &= q^{-s-2}L(2s,\omega)L(2s+1,\omega)^{-1}(\alpha-\beta)^{-1} \\
&\times \sum_{i=0}^{2}\alpha^{2i}\beta^{2i} q^{-2is}\sum_{n=0}^{\infty}q^{-2ns}(\alpha^{6n+2i+4}-\beta^{6n+2i+4})(1-(\alpha^3\beta^3q^{-2s})^{2n+2})\\
&= q^{-s-2}L(2s,\omega)L(2s+1,\omega)^{-1}(\alpha-\beta)^{-1}\\
&\times \Bigg[ \frac{1-\alpha^{12}\beta^6 q^{-6s}}{1-\alpha^4\beta^2 q^{-2s}}\cdot \left( \frac{\alpha^4}{1-\alpha^6 q^{-2s}} - \frac{\alpha^{10}\beta^6 q^{-4s}}{1-\alpha^{12} \beta^6 q^{-6s}}  \right)\\
&\quad\quad- \frac{1-\alpha^{6}\beta^{12} q^{-6s}}{1-\alpha^2\beta^4 q^{-2s}}\cdot \left( \frac{\beta^4}{1-\beta^6 q^{-2s}} - \frac{\alpha^{6}\beta^{10} q^{-4s}}{1-\alpha^{6} \beta^{12} q^{-6s}}  \right) \Bigg].
\end{align*}
By a direct calculation, we have
\begin{align*}
&\left( \frac{\alpha}{1-\alpha^6 q^{-2s}} - \frac{\alpha^4\beta^3 q^{-2s}}{1-\alpha^{12} \beta^6 q^{-6s}}  \right) + \left( \frac{\alpha^4q^{-s}}{1-\alpha^6 q^{-2s}} - \frac{\alpha^{10}\beta^6 q^{-5s}}{1-\alpha^{12} \beta^6 q^{-6s}}  \right)\\
&=L(2s,\omega)^{-1}\alpha(1-\alpha^3 q^{-s})^{-1}(1-\alpha^6\beta^3q^{-3s})^{-1},\\
&\left( \frac{\beta}{1-\beta^6 q^{-2s}} - \frac{\alpha^3\beta^4 q^{-2s}}{1-\alpha^{6} \beta^{12} q^{-6s}}  \right) + \left( \frac{\beta^4q^{-s}}{1-\beta^6 q^{-2s}} - \frac{\alpha^{6}\beta^{10} q^{-5s}}{1-\alpha^{6} \beta^{12} q^{-6s}}  \right)\\
&=L(2s,\omega)^{-1}\beta(1-\beta^3 q^{-s})^{-1}(1-\alpha^3\beta^6q^{-3s})^{-1}.
\end{align*}
Therefore, 
\begin{align*}
Z^{(0)}(s) + q^2 \cdot Z^{(1)}(s) & = L(2s+1,\omega)^{-1}(\alpha-\beta)^{-1}\\
& \times \left[ \frac{\alpha(1-\alpha^2\beta q^{-s} + \alpha^4 \beta^2 q^{-2s})}{(1-\alpha^3 q^{-s})(1-\alpha^2\beta q^{-s})} - \frac{\beta(1-\alpha\beta^2q^{-s}+\alpha^2\beta^4 q^{-2s})}{(1-\beta^3 q^{-s})(1-\alpha\beta^2q^{-s})}\right]\\
& = L(2s+1,\omega)^{-1}(\alpha-\beta)^{-1}\\
&\times \frac{(\alpha-\beta)(1-\alpha^6\beta^6 q^{-4s})}{(1-\alpha^3 q^{-s})(1-\beta^3 q^{-s})(1-\alpha^2\beta q^{-s})(1-\alpha\beta^2 q^{-s})}\\
&=L(2s+1,\omega)^{-1}L(4s,\omega^2)^{-1} L(s,{\rm As}\,\Pi).
\end{align*}
This completes the proof.
\end{proof}

\begin{corollary}\label{L:unramified gamma2}
Let $\Pi$ be an irreducible generic admissible representation of $\GL_2(E)$ with central character $\omega_\Pi$. Assume $\Pi$ is unramified.
We have
\[
\gamma_{\rm PSR}(s,{\rm As}\,\Pi,\psi,\alpha) = \omega_\Pi(\Delta_{E/F}(\alpha))|\Delta_{E/F}(\alpha)|_F^{2s-1}\omega_{\mathbb{K}/F}(-1)\gamma(s,{\rm As}\,\Pi,\psi)
\]
for any basis $\alpha$ of $E$ over $F$ and non-trivial additive character $\psi$ of $F$. 
\end{corollary}

\begin{proof}
We write $\omega = \omega_{\Pi}\vert_{F^\times}$. Assume 
$\Pi = {\rm Ind}_{B(E)}^{\GL_2(E)}(\chi_1\otimes\chi_2)$
for some unramified characters $\chi_1$ and $\chi_2$ of $E^\times$. We first prove the assertion for $\alpha = \alpha^o_X = \left\{(0,{1}/{3}), (0,{\varpi_E^{-1}}/{3}),(0,{\varpi_E^{-2}}/{3})  \right\}$ and $\psi$ is of conductor $\o_F$. Let $f_s^o$ be the $K$-invariant good section of $I(\omega,s)$ normalized so that $f_s^o(1)=1$ and $W^o \in \mathcal{W}(\Pi,\psi_E)$ the ${\bf a}(3^{-1}\varpi_E^{-2})\GL_2(\o_E){\bf a}(3\varpi_E^{2})$-invariant Whittaker function normalized so that $W^o(1)=1$. By the Gindikin-Karpelevich formula (cf.\,\cite[Proposition 5.3]{LNM1254}), we have
\[
M_{\rm w}^*f_s^o(1) = \frac{L(3-2s,\omega^{-1})L(4-4s,\omega^{-2})}{L(2s+1,\omega)L(4s,\omega^2)}.
\]
Note that $(W^o)^\vee \in \mathcal{W}(\Pi,\psi_E)$ is the ${\bf a}(3^{-1}\varpi_E^{-2})\GL_2(\o_E){\bf a}(3\varpi_E^{2})$-invariant Whittaker function with $(W^o)^\vee(1)=1$. By Lemma \ref{L:unram.}, we have
\begin{align*}
Z_{\alpha^o}(f_s^o,W^o) &= L(2s+1,\omega)^{-1}L(4s,\omega^2)^{-1} L(s,{\rm As}\,\Pi),\\
Z_{\alpha^o}(M_{\rm w}^*f_s^o,(W^o)^\vee) &= L(2s+1,\omega)^{-1}L(4s,\omega^2)^{-1} L(1-s,{\rm As}\,\Pi^\vee).
\end{align*}
It follows from the functional equation that
\[
\gamma_{\rm PSR}(s,{\rm As}\,\Pi,\psi,\alpha) = \varepsilon(s,{\rm As}\,\Pi,\psi)^{-1}\gamma(s,{\rm As}\,\Pi,\psi).
\]
Since $\chi_1$ and $\chi_2$ are unramified, by (\ref{E:Asai cubic 1}) and (\ref{E:Asai cubic 2}), we have
\begin{align*}
\varepsilon(s,{\rm As}\,\Pi,\psi) &= \varepsilon(s,\chi_1 \vert_{F^\times},\psi)\varepsilon(s,\chi_2 \vert_{F^\times},\psi)\varepsilon(s,{\rm Ind}_{W_E'}^{W_F'}(\chi_1^2 \chi_2),\psi)\varepsilon(s,{\rm Ind}_{W_E'}^{W_F'}(\chi_1 \chi_2^2),\psi).
\end{align*}
Also note that $\varepsilon(s,\chi_1 \vert_{F^\times},\psi)=\varepsilon(s,\chi_2 \vert_{F^\times},\psi)=1$
since $\psi$ is of conductor $\o_F$.
On the other hand, 
\begin{align*}
\varepsilon(s,{\rm Ind}_{W_E'}^{W_F'}(\chi_1^2 \chi_2),\psi) &= \lambda_{E/F}(\psi)\varepsilon(s,\chi_1^2 \chi_2,\psi_E),\\
\varepsilon(s,{\rm Ind}_{W_E'}^{W_F'}(\chi_1 \chi_2^2),\psi) &= \lambda_{E/F}(\psi)\varepsilon(s,\chi_1 \chi_2^2,\psi_E),
\end{align*}
where $\lambda_{E/F}(\psi)$ is the Langlands constant for $E/F$ with respect to $\psi$ (cf.\,\cite[\S\,30]{BH2006}).
Since $\chi_1$ and $\chi_2$ are unramified and $\psi_E^{3^{-1}\varpi_E^{-2}}$ is of conductor $\o_E$, we have
\begin{align*}
\varepsilon(s,\chi_1^2 \chi_2,\psi_E) &= \chi_1^2\chi_2(3\varpi_E^{2})|3\varpi_E^{2}|_E^{s-1/2},\\
\varepsilon(s,\chi_1 \chi_2^2,\psi_E) &= \chi_1\chi_2^2(3\varpi_E^{2})|3\varpi_E^{2}|_E^{s-1/2}.
\end{align*}
Note that $\Delta_{E/F}(\alpha) = 3^{-3} \varpi_F^{-2}$ and $\lambda_{E/F}(\psi)^2 = \omega_{\mathbb{K}/F}(-1)$. Indeed, since $\omega_{\mathbb{K}/F} = \det({\rm Ind}_{W_E'}^{W_F'}(1))$, we have $\lambda_{E/F}(\psi)^2 = \omega_{\mathbb{K}/F}(-1)$ by \cite[(30.4.3)]{BH2006}. We conclude that
\[
\varepsilon(s,{\rm As}\,\Pi,\psi) = \omega_\Pi(\Delta_{E/F}(\alpha))^{-1}|\Delta_{E/F}(\alpha)|_F^{-2s+1}\omega_{\mathbb{K}/F}(-1).
\]
The assertion for the special case is proved. It then follows from (\ref{E:basic properties WD}) and Lemma \ref{L:basic properties zeta}-(2) that the assertion holds for any non-trivial additive character $\psi$ of $F$ and $\alpha = \alpha^o_X$. Let $\beta$ be any basis of $E$ over $F$ and $A_{\alpha,\beta} \in \GL_3(F)$ the transition matrix from $\alpha$ to $\beta$. Then $\Delta_{E/F}(\beta) = \det(A_{\alpha,\beta})^2\Delta_{E/F}(\alpha)$. The assertion for $\beta$ then follows from Lemma \ref{L:basic properties zeta}-(1). This completes the proof.
\end{proof}

\section{Proof of Main Results}

\subsection{Automorphic $L$-functions}\label{S:auto. L}
In \S\,\ref{SS:proof 1} below, a key ingredient in the proof of Theorem \ref{T:1} is a global-to-local argument. We recall in this section the automorphic $L$-functions for the Asai cube representation. Let $F$ be a number field and $E$ an \'etale cubic algebra over $F$. Let ${\mathbb K}$ be the quadratic discriminant algebra of $E$ and $\omega_{{\mathbb K}/F}$ the quadratic Hecke character associated to ${\mathbb K}/F$ by class field theory. Denote by $\A_E$ and $\A_F$ the rings of adeles of $E$ and $F$, respectively. Let $\psi$ be a non-trivial additive character of $\A_F/F$. Let $\Pi$ be an irreducible cuspidal automorphic representation of $\GL_2(\A_E)$ with central character $\omega_{\Pi}$. Write $\omega = \omega_\Pi \vert_{F^\times}$. For each place $v$ of $F$, let 
\[
L(s,{\rm As}\,\Pi_v),\quad \varepsilon(s,{\rm As}\,\Pi_v,\psi_v),\quad \gamma(s,{\rm As}\,\Pi_v,\psi_v)
\]
be the Asai cube factors associated to $\Pi_v$ defined in \S\,\ref{S:Asai cubic Galois} and 
\[
\gamma_{\rm LS}(s, {\rm As}\,\Pi_v,\psi_v)
\]
the Asai cube $\gamma$-factor defined by the Langlands-Shahidi method \cite{Shahidi1990} (cf.\,\cite[\S\,5]{HL2018}). By \cite[Theorem 3.5-(1)]{Shahidi1990} for $v$ archimedean and \cite[Theorem 1.1]{HL2018} for $v$ non-archimedean (see also \cite[Theorem 4.4.1]{Rama2000}, \cite[\S\,5]{Kim2003}, and \cite[\S\,6]{Kris2003}), we have
\begin{align}\label{E:equality gamma factor}
\gamma (s, {\rm As}\,\Pi_v,\psi_v)= \gamma_{\rm LS}(s, {\rm As}\,\Pi_v,\psi_v).
\end{align}
The automorphic $L$-function  and $\varepsilon$-factor associated to $\Pi$ and the Asai cube representation are defined by
\[
L(s,{\rm As}\,\Pi) = \prod_{v}L(s,{\rm As}\,\Pi_v), \quad \varepsilon(s,{\rm As}\,\Pi) = \prod_v\varepsilon(s,{\rm As}\,\Pi_v,\psi_v).
\]
Note that the product defining the $L$-function is absolutely convergent for ${\rm Re}(s)$ sufficiently large. For a set $S$ of places of $F$, we define the partial $L$-function
\[
L^S(s,{\rm As}\,\Pi) = \prod_{v \notin S}L(s,{\rm As}\,\Pi_v).
\]

We recall the integral representation of $L(s,{\rm As}\,\Pi)$ due to Garrett in \cite{Garrett1987} (see also \cite{PSR1987} and \cite{Ikeda1989}). Let $\alpha$ be a symplectic basis of $V(F)$ and 
\[
\iota_{\alpha,\A_F} : \GSp(V)(\A_F) \longrightarrow \GSp_6(\A_F)
\]
the isomorphism defined by $\iota_{\alpha,\A_F} = \bigotimes_v \iota_{\alpha,v}$, where 
\[
\iota_{\alpha,v} : \GSp(V)(F_v) \longrightarrow \GSp_6(F_v)
\]
is the isomorphism (\ref{E:iota}) by regarding $\alpha$ as a symplectic basis of $V(F_v)$. Let $f_s$ be a good section of $I(\omega,s)$ and $\phi \in \Pi$ a cusp form. Let $E(f_s)$ be the Eisenstein series on $\GSp_6(\A_F)$ defined by
\[
E(g;f_s) = \sum_{\gamma \in P(F) \backslash \GSp_6(F)}f_s(\gamma g)
\]
for ${\rm Re}(s)$ sufficiently large, and by meromorphic continuation otherwise. Let $W_\phi$ be the Whittaker function of $\phi$ with respect to $\psi_E$ defined by
\[
W_\phi(g) = \int_{E \backslash \A_E}\phi({\bf n}(x)g)\overline{\psi_E(x)}\,dx.
\]
By a standard unfolding argument (cf.\,\cite[\S\,2]{PSR1987}), the integral
\[
\int_{\A_F^\times \G(F) \backslash \G(\A_F)}E(\iota_{\alpha,\A_F}(g);f_s)\phi(g)\,dg
\]
is equal to 
\[
Z_\alpha(f_s,W_\phi) = \int_{\A_F^\times U_0(\A_F)\backslash \G(\A_F)}f_s(\iota_{\alpha,\A_F}(\eta g))W_\phi(g)\,dg
\]
for ${\rm Re}(s)$ sufficiently large, where $\eta \in \GSp(V)(F)$ satisfies (\ref{E:eta}).
By the functional equation of the Eisenstein series $E(f_s)$ and the properties of the local zeta integrals in \S\,\ref{SS:local factor}, the integral $Z_\alpha(f_s,W_\phi)$ admits meromorphic continuation to $s \in \C$ and satisfies the functional equation
\begin{align}\label{E:global fe}
Z_\alpha(f_s,W_\phi) = Z_\alpha(M_{\rm w}^*f_s,W_\phi^\vee).
\end{align}

\begin{lemma}\label{L:crude gamma}
Let $S$ be a finite set of places of $F$ containing all archimedean places such that for $v \notin S$, we have $E_v$ is unramified over $F_v$ and $\Pi_v$ is unramified.
Then 
\[
\prod_{v \in S}\gamma_{\rm PSR}(s,{\rm As}\,\Pi_v,\psi_v,\alpha) = \prod_{v \in S}\omega_{\Pi_v}(\Delta_{E_v/F_v}(\alpha))|\Delta_{E_v/F_v}(\alpha)|_{F_v}^{2s-1}\omega_{\mathbb{K}_v/F_v}(-1)\gamma(s,{\rm As}\,\Pi_v,\psi_v)
\]
for any basis $\alpha$ of $E$ over $F$.
\end{lemma}

\begin{proof}
Let $\alpha$ be a basis of $E$ over $F$. Let $T$ be a set of places containing $S$ such that for $v \notin T$, we have
\begin{itemize}
\item $\alpha$ is an integral basis of $\o_{E_v}$ over $\o_{F_v}$,
\item $\psi_v$ is of conductor $\o_{F_v}$,
\item $E_v$ is unramified over $F_v$,
\item $\Pi_v$ is unramified.
\end{itemize}
Write $\alpha = \{x_1,x_2,x_3\}$ and let $\alpha^*=\{x_1^*,x_2^*,x_3^*\}$ be the dual basis of $\alpha$. Let $\tilde{\alpha}$ be the symplectic basis of $V(F)$ defined by
\[
\tilde{\alpha} = \{(x_1^*,0),(x_2^*,0),(x_3^*,0),(0,x_1),(0,x_2),(0,x_3)\}.
\]
For each place $v$ of $F$, let $f_{s,v}$ be a good section of $I(\omega_v,s)$ and $W_v \in \mathcal{W}(\Pi_v,\psi_{E,v})$ satisfying the following conditions:
\begin{itemize}
\item For $v \notin T$, $f_{s,v}$ is the $\GSp_6(\o_{F_v})$-invariant good section normalized so that $f_{s,v}(1)=1$ and $W_v$ the $\GL_2(\o_{E_v})$-invariant Whittaker function normalized so that $W_v(1)=1$.
\item For $v \in T$, we have
$Z_{\tilde{\alpha}}(f_{s,v},W_v) \neq 0$.
\end{itemize}
By \cite[Theorem 3.1]{PSR1987}, for $v \notin T$ we have
\begin{align*}
Z_{\tilde{\alpha}}(f_{s,v},W_v) &= L(2s+1,\omega_v)^{-1}L(4s,\omega_v^2)^{-1} L(s,{\rm As}\,\Pi_v),\\
Z_{\tilde{\alpha}}(M_{\rm w}^*f_{s,v},W_v^\vee) &= L(2s+1,\omega_v)^{-1}L(4s,\omega_v^2)^{-1} L(1-s,{\rm As}\,\Pi_v^\vee).
\end{align*}
We conclude from the functional equations (\ref{E:fe}) and (\ref{E:global fe}) that the partial $L$-function $L^T(s,{\rm As}\,\Pi)$ admits meromorphic continuation to $s \in \C$ and satisfies the functional equation
\[
L^T(s,{\rm As}\,\Pi) = \prod_{v \in T}\gamma_{\rm PSR}(s,{\rm As}\,\Pi_v,\psi_v,\alpha)L^T(1-s,{\rm As}\,\Pi^\vee).
\]
On the other hand, by \cite[Theorem 3.5-(4)]{Shahidi1990} and (\ref{E:equality gamma factor}), we have the functional equation
\[
L^T(s,{\rm As}\,\Pi) = \prod_{v \in T}\gamma(s,{\rm As}\,\Pi_v,\psi_v)L^T(1-s,{\rm As}\,\Pi^\vee).
\]
We deduce that
\begin{align}\label{E:gamma equality}
\prod_{v \in T}\gamma_{\rm PSR}(s,{\rm As}\,\Pi_v,\psi_v,\alpha) = \prod_{v \in T}\gamma(s,{\rm As}\,\Pi_v,\psi_v).
\end{align}
By Lemma \ref{L:unramified gamma}, for $v \in T\setminus S$, we have
\[
\gamma_{\rm PSR}(s,{\rm As}\,\Pi_v,\psi_v,\alpha) = \omega_{\Pi_v}(\Delta_{E_v/F_v}(\alpha))|\Delta_{E_v/F_v}(\alpha)|_{F_v}^{2s-1}\omega_{\mathbb{K}_v/F_v}(-1)\gamma(s,{\rm As}\,\Pi_v,\psi_v).
\]
By assumption, for $v \notin T$, we have
\[
\omega_{\Pi_v}(\Delta_{E_v/F_v}(\alpha))|\Delta_{E_v/F_v}(\alpha)|_{F_v}^{2s-1}\omega_{\mathbb{K}_v/F_v}(-1)=1.
\]
The assertion then follows from (\ref{E:gamma equality}) and the product formula
\[
\prod_v \omega_{\Pi_v}(\Delta_{E_v/F_v}(\alpha))|\Delta_{E_v/F_v}(\alpha)|_{F_v}^{2s-1}\omega_{\mathbb{K}_v/F_v}(-1)=1.
\]
This completes the proof.
\end{proof}

\subsection{Proof of Theorem \ref{T:1}}\label{SS:proof 1}
Let $F$ be a non-archimedean local field of characteristic zero and $E$ an \'etale cubic algebra over $F$. Let ${\mathbb K}$ be the quadratic discriminant algebra of $E$. Let $\Pi$ be an irreducible generic admissible representation of $\GL_2(E)$ with central character $\omega_\Pi$. Let $\psi$ be a non-trivial additive character of $F$ and $\alpha$ a basis of $E$ over $F$. Recall the domains $\mathcal{D}(\Pi)$ and $\mathcal{D}(\Pi)^\circ$ defined in (\ref{E:domain}) and (\ref{E:imaginary domain}), respectively. Consider the family of irreducible generic admissible representations $\Pi_\lambda$ of $\GL_2(E)$ for $\lambda$ varying in $\mathcal{D}(\Pi)$. We are going to prove that the identity 
\begin{align}\label{E:main identity}
\gamma_{\rm PSR}(s,{\rm As}\,\Pi_\lambda,\psi,\alpha) = \omega_{\Pi_\lambda}(\Delta_{E/F}(\alpha))|\Delta_{E/F}(\alpha)|_F^{2s-1}\omega_{\mathbb{K}/F}(-1)\gamma(s,{\rm As}\,\Pi_\lambda,\psi)
\end{align}
holds for all $\lambda \in \mathcal{D}(\Pi)$. In particular, (\ref{E:main identity}) holds for $\Pi_\lambda = \Pi$. We divide the proof into three steps as follows:
\begin{itemize}
\item[Step 1.] Establish (\ref{E:main identity}) for a dense subset $\mathcal{U}$ of $\mathcal{D}(\Pi)^\circ$.
\item[Step 2.] Establish (\ref{E:main identity}) for all $\lambda \in \mathcal{D}(\Pi)$ with $|\lambda|_\Pi<1/2$.
\item[Step 3.] Establish (\ref{E:main identity}) for all $\lambda \in \mathcal{D}(\Pi)$.
\end{itemize}
Fix a symplectic basis $\tilde{\alpha}$ of $V(F)$ such that $\tilde{\alpha}_X = \alpha$.  Write $\omega_\lambda = \omega_{\Pi_\lambda}\vert_{F^\times}$.

{\bf Step\,1.}  
By Kranser's lemma, there exist an \'etale cubic algebra ${\bm E}$ over a number field $\bm F$ and a place $v_1$ of $\bm F$ such that 
\begin{itemize}
\item ${\bm E}_{v_1}=E$ and ${\bm F}_{v_1}=F$,
\item for any non-archimedean place $v \neq v_1$ and $v$ divides $3$, ${\bm E}_v$ is unramified over ${\bm F}_v$.
\end{itemize}
Let ${\bm \alpha}$ be a basis of ${\bm E}$ over $\bm F$, $\bm K$ the quadratic discriminant algebra of $\bm E$, and $\bm \psi$ a non-trivial additive character of $\A_{\bm F} / {\bm F}$. By (\ref{E:basic properties WD}) and Lemma \ref{L:basic properties zeta}, we may assume $\alpha = {\bm \alpha}$ regarding as a basis of $E$ over $F$ and $\psi = {\bm \psi}_v$. Fix a non-archimedean place $v_2$ of $\bm F$ such that ${\bm E}_{v_2} = {\bm F}_{v_2} \times {\bm F}_{v_2} \times {\bm F}_{v_2}$. Let $U$ be an open subset of $\mathcal{D}(\Pi)^\circ$. By the limit multiplicity property for the principal congruence subgroups of $\GL_2$ proved in \cite[Theorem 1.3]{FLM2015}, there exists an irreducible cuspidal automorphic representation ${\bm \Pi}$ of $\GL_2(\A_{\bm E})$ (see \cite[Theorem 3.7.1]{Raphael2018} for the deduction) such that
\begin{itemize}
\item ${\bm \Pi}_{v_1} = \Pi_\lambda$ for some $\lambda \in U$,
\item ${\bm \Pi}_v$ is unramified for all non-archimedean places $v \notin \{v_1,v_2\}$.
\end{itemize}
Let $S$ be the finite set of places $v$ of $\bm F$ such that $v$ is archimedean or $v \in \{v_1,v_2\}$ or $v$ is non-archimedean and ${\bm E}_v$ is ramified over ${\bm F}_v$. By Lemma \ref{L:crude gamma}, we have
\[
\prod_{v \in S}\gamma_{\rm PSR}(s,{\rm As}\,{\bm \Pi}_v,{\bm \psi}_v,{\bm \alpha}) = \prod_{v \in S}\omega_{{\bm \Pi}_v}(\Delta_{{\bm E}_v/{\bm F}_v}({\bm \alpha}))|\Delta_{{\bm E}_v/{\bm F}_v}({\bm \alpha})|_{{\bm F}_v}^{2s-1}\omega_{{\bm K}_v/{\bm F}_v}(-1)\gamma(s,{\rm As}\,{\bm \Pi}_v,{\bm \psi}_v).
\]
On the other hand, by Lemma \ref{L:unramified gamma} and Corollary \ref{L:unramified gamma2}, for $v \in S$ with $v \neq v_1$, we have
\[
\gamma_{\rm PSR}(s,{\rm As}\,{\bm \Pi}_v,{\bm \psi}_v,{\bm \alpha}) = \omega_{{\bm \Pi}_v}(\Delta_{{\bm E}_v/{\bm F}_v}({\bm \alpha}))|\Delta_{{\bm E}_v/{\bm F}_v}({\bm \alpha})|_{{\bm F}_v}^{2s-1}\omega_{{\bm K}_v/{\bm F}_v}(-1)\gamma(s,{\rm As}\,{\bm \Pi}_v,{\bm \psi}_v).
\]
It follows that (\ref{E:main identity}) holds for some $\lambda \in U$.

{\bf Step\,2.} Let $\lambda_0 \in \mathcal{D}(\Pi)$ with $|\lambda_0|_\Pi < 1/2$. Let $0<\epsilon< 1/2-|\lambda_0|_\Pi$. Let $f_s$ be a good section of $I(\omega_{\lambda_0},s)$ and $W \in \mathcal{W}(\Pi_{\lambda_0},\psi)$. We extend $f_s$ to a good section $f_{s,\lambda}$ of $I(\omega_\lambda,s)$ and $W$ to a holomorphic family of Whittaker functions $W_\lambda$ of $\Pi_\lambda$ with respect to $\psi_E$. By {\bf Step\,1},
\begin{align}\label{E:5.2}
\frac{Z_{\tilde\alpha}(M_{\rm w}^*f_{s,\lambda},W_\lambda^\vee)}{L(1-s,{\rm As}\,\Pi_\lambda^\vee)} = \omega_{\Pi_\lambda}(\Delta_{E/F}(\alpha))|\Delta_{E/F}(\alpha)|_F^{2s-1}\omega_{\mathbb{K}/F}(-1)\varepsilon(s,{\rm As}\,\Pi_\lambda,\psi,\alpha)\frac{Z_{\tilde \alpha}(f_{s,\lambda},W_\lambda)}{L(s,{\rm As}\,\Pi_\lambda)}
\end{align}
holds for $s \in \C$ and $\lambda \in \mathcal{U}$. On the other hand, by Lemma \ref{L:convergence PSR}, the right-hand side and the left-hand side of (\ref{E:5.2}) define holomorphic functions on 
\[
\{s\in \C\,\vert\, {\rm Re}(s)>1/2-\epsilon\} \times \{\lambda \in \mathcal{D}(\Pi)\,\vert\, |\lambda|_\Pi < 1/2-\epsilon\}
\]
and 
\[
\{s\in \C\,\vert\, {\rm Re}(s)<1/2+\epsilon\} \times \{\lambda \in \mathcal{D}(\Pi)\,\vert\, |\lambda|_\Pi < 1/2-\epsilon\},
\]
respectively.
Since the set $\mathcal{U}$ is dense in $\mathcal{D}(\Pi)^\circ$, it follows that (\ref{E:5.2}) holds for $(s,\lambda)$ in the domain
\[
\{s\in \C\,\vert\, 1/2-\epsilon<{\rm Re}(s)<1/2+\epsilon\} \times \{\lambda \in \mathcal{D}(\Pi)\,\vert\, |\lambda|_\Pi < 1/2-\epsilon\}.
\]
We then deduce form the holomorphicity that (\ref{E:5.2}) holds for all $s \in \C$ and $|\lambda|_\Pi < 1/2-\epsilon$. In particular, it holds for $\lambda = \lambda_0$. Since $f_s$ and $W$ are arbitrary chosen, we see that $(\ref{E:main identity})$ holds for $\lambda = \lambda_0$. 

{\bf Step\,3.} Let $f_{s,\lambda}$ be a good section of $I(\omega_\lambda,s)$ and $W_\lambda$ a holomorphic family of Whittaker functions of $\Pi_\lambda$ with respect to $\psi_E$.  Let 
\[
\mathcal{Z}_1(s,\lambda) = \frac{Z_{\tilde\alpha}(M_{\rm w}^*f_{-s+1/2,\lambda},W_\lambda^\vee)}{L(s+1/2,{\rm As}\,\Pi_\lambda^\vee)}
\]
and
\[
\mathcal{Z}_2(s,\lambda) = \omega_{\Pi_\lambda}(\Delta_{E/F}(\alpha))|\Delta_{E/F}(\alpha)|_F^{2s}\omega_{\mathbb{K}/F}(-1)\varepsilon(s+1/2,{\rm As}\,\Pi_\lambda,\psi,\alpha)\frac{Z_{\tilde \alpha}(f_{s+1/2,\lambda},W_\lambda)}{L(s+1/2,{\rm As}\,\Pi_\lambda)}
\]
be partially defined functions on $\C \times \mathcal{D}(\Pi)$.
By Lemma \ref{L:convergence PSR} and {\bf Step\,2}, $\mathcal{Z}_1$ and $\mathcal{Z}_2$ are holomorphic functions on the domain
\[
\C \times \{\lambda \in \mathcal{D}(\Pi)\,\vert\,|\lambda|_\Pi<1/2\}
\]
and satisfy the functional equation
\[
\mathcal{Z}_1(-s,\lambda) = \mathcal{Z}_2(s,\lambda).
\]
Note that it is clear that $\mathcal{Z}_1$ and $\mathcal{Z}_2$ are periodic of period $\log(q_F)^{-1}2\pi\sqrt{-1}$ in the variable $s$. We claim that $\mathcal{Z}_1$ and $\mathcal{Z}_2$ admit holomorphic continuation to $\C \times \mathcal{D}(\Pi)$ and satisfy the above functional equation. To prove the claim, by \cite[Proposition 2.8.1]{Raphael2018} with $M=\mathcal{D}(\Pi)$ and $U = \{\lambda \in \mathcal{D}(\Pi)\,\vert\,|\lambda|_\Pi<1/2\}$, it suffices to show that for every $C' \geq 1/2$, there exists $C>0$ such that $\mathcal{Z}_1$ and $\mathcal{Z}_2$ admit holomorphic continuation to the domain
\[
\{s\in \C\,\vert\, {\rm Re}(s)>C\} \times \{\lambda \in \mathcal{D}(\Pi)\,\vert\, |\lambda|_\Pi < C'\}.
\]
By Lemma \ref{L:convergence PSR}, we can take $C = C'-1/2$. As $f_{s,\lambda}$ and $W_\lambda$ are arbitrary chosen, we conclude that (\ref{E:main identity}) holds for all $\lambda \in \mathcal{D}(\Pi)$. This completes the proof of Theorem \ref{T:1}. 

\begin{rmk}
In \cite[Proposition 2.8.1]{Raphael2018}, $U$ and $U'$ are connected relatively compact open subsets of $M$. It is clear that any connected relatively compact open subset of $\mathcal{D}(\Pi)$ is contained in $\{\lambda \in \mathcal{D}(\Pi)\,\vert\, |\lambda|_\Pi < C\}$ for some $C>0$.
\end{rmk}

\subsection{Proof of Corollary \ref{C:1}}
By \cite[Lemma 2.1]{Ikeda1992} and Lemma \ref{L:convergence WD}, $L(s,{\rm As}\,\Pi)$ and $L_{\rm PSR}(s,{\rm As}\,\Pi)$ have no poles for ${\rm Re}(s)\geq1/2$. Similarly, $L(1-s,{\rm As}\,\Pi^\vee)$ and $L_{\rm PSR}(1-s,{\rm As}\,\Pi^\vee)$ have no poles for ${\rm Re}(s)\leq1/2$. Also note that the $L$-functions have no zeros. We deduce that the poles of $L(s,{\rm As}\,\Pi)$, count with multiplicities, are equal to the poles of the meromorphic function $L(s,{\rm As}\,\Pi) L(1-s,{\rm As}\,\Pi^\vee)^{-1}$. We have a similar conclusion for $L_{\rm PSR}(s,{\rm As}\,\Pi)$. Since the $\varepsilon$-factors have neither poles nor zeros, the assertion then follows from Theorem \ref{T:1}.

\subsection{Proof of Theorem \ref{T:2}}
By \cite[Theorem C]{KS2002}, we have $\max\{|L(\Pi_v)| , |L(\Pi_v^\vee)|\} < 1/2$ for all places $v$ of $F$. Therefore, by Corollary \ref{C:1}, we have $L_{\rm PSR}(s,{\rm As}\,\Pi_v) = L(s,{\rm As}\,\Pi_v)$ for all places $v$ of $F$. It follows from \cite[Propositions 2.3-2.5]{Ikeda1992} and the integral representation recalled in \S\,\ref{S:auto. L} that $L(s,{\rm As}\,\Pi)$ is absolutely convergent for ${\rm Re}(s)\geq3/2$, admits meromorphic continuation to $s \in \C$, entire if either $\omega^2 \neq 1$ or $\omega = 1$, and satisfies the functional equation
\[
L(s,{\rm As}\,\Pi) = \varepsilon(s,{\rm As}\,\Pi) L(1-s,{\rm As}\,\Pi^\vee).
\]
Proceeding as in the proof of \cite[Theorem 3.4.1]{Rama2000}, we can prove that $L(s,{\rm As}\,\Pi)$ is bounded in vertical strips of finite width. Note that in our case the assumption that $F$ is totally complex in \cite[Theorem 3.4.1]{Rama2000} is not necessary. Finally, the description of the poles at $s=0,1$ were established in \cite[Theorems 2.6-2.8]{Ikeda1992}.

\bibliographystyle{alpha}
\bibliography{ref}	
\end{document}